\documentclass[11pt]{amsart}
\usepackage[utf8]{inputenc}
\usepackage{amsfonts, amssymb, amsmath, amsthm, color, float,enumerate}
\usepackage{url}
\usepackage{bm}
\usepackage[unicode,psdextra]{hyperref}
\usepackage{pgfplots,tikz}
\pgfplotsset{compat=1.18}

\title[Smoothness properties of multilinear commutators]{Smoothness properties related to several commutators of fractional operators for values of $\boldsymbol{p}$ beyond the extreme in the multilinear setting}
\author{}
\date{}
\usepackage[a4paper, left=2cm, right=2cm, top=3cm, bottom=3cm]{geometry} 

\theoremstyle{plain}
   \newtheorem{teo}{Theorem}
   \newtheorem{coro}[teo]{Corollary}
   \newtheorem{lema}[teo]{Lemma}
   \newtheorem{propo}[teo]{Proposition}
   
\theoremstyle{definition}
   
\theoremstyle{remark}
 \newtheorem{obs}{Remark}

\numberwithin{equation}{section}
\numberwithin{teo}{section}
\allowdisplaybreaks

\definecolor{aquamarine}{rgb}{0.5, 1.0, 0.83}
\definecolor{americanrose}{rgb}{1.0, 0.01, 0.24}
\definecolor{arsenic}{rgb}{0.23, 0.27, 0.29}
\definecolor{blizzardblue}{rgb}{0.67, 0.9, 0.93}
\definecolor{blush}{rgb}{0.87, 0.36, 0.51}
\definecolor{celestialblue}{rgb}{0.29, 0.59, 0.82}
\definecolor{chocolate(web)}{rgb}{0.82, 0.41, 0.12}
\definecolor{brightpink}{rgb}{1,0,0.5}
\definecolor{cadmiunred}{rgb}{0.89,0,0.13}
\definecolor{darkpastelpurple}{rgb}{0.59, 0.44, 0.84}
\definecolor{atomictangerine}{rgb}{1.0, 0.6, 0.4}

\hypersetup{
	colorlinks = true,
	linkcolor = cadmiunred,
	anchorcolor = blue,
	citecolor = brightpink,
	filecolor = blue,
	urlcolor = blue
}

\newcounter{BPR}

\begin{document}

\author[F. Berra]{Fabio Berra}
\address{Fabio Berra, CONICET and Departamento de Matem\'{a}tica (FIQ-UNL),  Santa Fe, Argentina.}
\email{fberra@santafe-conicet.gov.ar}
	
\author[W. Ramos]{Wilfredo Ramos}
\address{CONICET and Departamento de Matem\'{a}tica (FaCENA-UNNE),  Corrientes, Argentina.}
\email{wilfredo.ramos@comunidad.unne.edu.ar}
	

 
\thanks{The author were supported by CONICET, UNL, ANPCyT and UNNE}
	


\maketitle

\begin{abstract}
We prove continuity properties of higher order commutators of  fractional operators on the multilinear setting, between a product of weighted Lebesgue spaces into certain weighted Lipschitz spaces.
The considered operators include the multilinear fractional integral function and the main results extend previous estimates known for the unweighted case, as well as those established on the linear case. 

We give a complete study showing the main properties of the related multilinear weights and the optimal region described by the parameters where we can find nontrivial examples, including the restricted case of the one-weight theory. We also exhibit examples of such weights on the whole region.
\end{abstract}

\section{Introduction}\label{seccion: introduccion}

In 1974, B. Muckenhoupt and R. Wheeden characterized the weight functions $w$ for which the strong $(p,q)$ inequality
\[\left(\int_{\mathbb{R}^n}(M_\alpha f)^qw^q\right)^{1/q}\leq C\left(\int_{\mathbb{R}^n}|f|^pw^p\right)^{1/p}\]
holds, where $M_\alpha$ is the classical fractional maximal operator given by
\[M_\alpha f(x)=\sup_{Q\ni x}\frac{1}{|Q|^{1-\alpha/n}}\int_Q |f|\]
and $1/q=1/p-\alpha/n$ (see \cite{Muckenhoupt-Wheeden74}). The weights involved belong to the so-called $A_{p,q}$ classes, which turn out to be a variant of the Muckenhoupt $A_p$ weights related to the strong $(p,p)$ problem for the classical Hardy-Littlewood maximal function. As well as for $M_\alpha$, they proved the same $(p,q)$ strong estimate for the fractional integral operator $I_\alpha$. Although this operator fails to have a strong $(n/\alpha,\infty)$ type, it was shown in \cite{Muckenhoupt-Wheeden74} that it maps $L^{n/\alpha}(w^{n/\alpha})$ into a weighted BMO space. When $n/\alpha<p<n/(\alpha-1)^+$,  $I_\alpha$ maps $L^p(w^p)$ into a suitable Lipschitz-$\delta$ space where the relation $\delta/n=\alpha/n-1/p$ holds, as it was proved in \cite{Pradolini01}. In this article a two-weighted problem was also given, as well as the region described by the parameters involved such that the considered classes of weights are nontrivial. 
 For the two-weighted theory we also refer the reader to \cite{SW92} and \cite{Pradolini01}.

On the other hand, the study of multilinear operators caught great attention over the last years. Perhaps a big step into this topic was done in \cite{L-O-P-T-T}, where a weighted theory for certain multilinear version of the Hardy-Littlewood maximal function and Calderón-Zygmund operators was introduced. The multilinear classes of weights are extensions of the classical $A_p$ functions, and their main properties were also established. This set a point of departure from which the multilinear theory opened several directions of study. For example, the characterization of multilinear weights in order to have strong type inequalities for the multilinear fractional integral operator was given in \cite{Moen09} (see also \cite{Pradolini10}). The outcoming classes of weights turned out to be multilinear extensions of the classical $A_{p,q}$ classes  and also power and logarithmic bumps for the case of two weights estimates.

In \cite{BPR22} we studied boundedness properties of the multilinear fractional integral operator $I_\alpha^m$ acting between a product of weighted Lebesgue spaces into the Lipschitz space $\mathbb{L}_w(\delta)$, defined as the collection of $f\in L^1_{\text{loc}}$ such that $\|f\|_{\mathbb{L}_w(\delta)}<\infty$, where
\[\|f\|_{\mathbb{L}_w(\delta)}=\sup_{B\subset \mathbb{R}^n}\frac{\|w\mathcal{X}_B\|_\infty}{|B|^{1+\delta/n}}\int_{B}|f-f_B|,\]
being $f_B$ the average of $f$ over $B$.
Concretely, we proved that $I_\alpha^m$ maps $\prod_{i=1}^mL^{p_i}(v^{p_i})$ into the space $\mathbb{L}_w(\delta)$ if and only if
there exists a positive constant $C$ such that the inequality
\begin{equation}\label{eq: clase Hbb(p, alfa, delta)}
\frac{\|w\mathcal{X}_B\|_\infty}{|B|^{(\delta-1)/n}}\prod_{i=1}^m\left\|\frac{v_i^{-1}}{(|B|^{1/n}+|x_B-\cdot|)^{n-\alpha_i+1/m}}\right\|_{p_i'}\leq C
\end{equation}
holds for every ball $B=B(x_B,R)$, where $0<\alpha_i<n$ for every $i$ and $\sum_{i=1}^m\alpha_i=\alpha$. In this case we say that
 the pair $(w,\vec{v})$ belongs to the class $\mathbb{H}_m(\vec{p},\alpha,\delta)$.
 When $m=1$, this result generalizes the two-weighted problem for the linear case given in \cite{Pradolini01} and also the unweighted estimates for $I_\alpha^m$ established in \cite{AHIV}. 
 
 A similar problem was established in \cite{BPR22(2)}, where we considered the Lipschitz class given by the collection of locally integrable functions $f$ such that $\|f\|_{\mathcal{L}_w(\delta)}<\infty$, where
 \[\|f\|_{\mathcal{L}_w(\delta)}=\sup_{B\subset \mathbb{R}^n}\frac{1}{w^{-1}(B)|B|^{\delta/n}}\int_{B}|f-f_B|.\]
In this context we characterized the pairs $(w,\vec{v})$ for which
$I_\alpha^m$ maps $\prod_{i=1}^mL^{p_i}(v^{p_i})$ into the space $\mathcal{L}_w(\delta)$ as those belonging to the class $\mathcal{H}_m(\vec{p},\alpha,\delta)$, that is, the pairs satisfying the inequality 
\begin{equation}\label{eq: clase Hcal(p, alfa, delta)}
\frac{1}{w^{-1}(B)|B|^{\delta/n}}\prod_{i=1}^m\left\|\frac{v_i^{-1}}{(|B|^{1/n}+|x_B-\cdot|)^{n-\alpha_i+1/m}}\right\|_{p_i'}\leq C
\end{equation}
for some positive constant $C$ and every ball $B=B(x_B,R)$. This last result extend the two-weighted estimates given for the linear case in \cite{Prado01cal} as well as the one weight boundedness for $I_\alpha$ proved in \cite{HSV}.

At a first glance we can observe that $\mathbb{H}_m(\vec{p},\gamma,\delta)\subset \mathcal{H}_m(\vec{p},\gamma,\delta)$. Furthermore, when $w\in A_1$ both classes coincide since in this case we have $\|w\mathcal{X}_B\|_\infty \approx \frac{|B|}{w^{-1}(B)}$. The same relation holds between the respective Lipschitz classes $\mathbb{L}_w(\delta)$ and $\mathcal{L}_w(\delta)$. Despite this fact, however, 
these classes of weights are substantially different. For instance, when we restrict our attention to the case $w=\prod_{i=1}^mv_i$, which is the multilinear generalization of the one weight case for $m=1$, we can see that under certain conditions \eqref{eq: clase Hbb(p, alfa, delta)} is equivalent to the $A_{\vec{p},\infty}$ class, whilst from \eqref{eq: clase Hcal(p, alfa, delta)} we can deduce that certain power of $w^{-1}$ belongs to a reverse Hölder class. Moreover, if we compare the regions described by the parameters involved where we find nontrivial weights, we shall see that the contention claimed above is strict. These properties were established in \cite{BPR22} and \cite{BPR22(2)}, respectively.

 Recently, in \cite{BPRe} the authors studied similar estimates as above for commutators of fractional operators which generalize $I_\alpha^m$. The corresponding class of weights is a variant of \eqref{eq: clase Hbb(p, alfa, delta)}. 

We devote this article to study continuity properties for commutators of fractional operators acting from a product of weighted Lebesgue spaces into the Lipschitz space $\mathcal{L}_w(\delta)$. The corresponding operators include the multilinear fractional integral function $I_{\alpha}^m$. Our aim is to obtain boundedness results for the operators above for a wider class of pairs of weights that contain those given in  \cite{BPRe}.

We begin by introducing the multilinear weights we shall consider throughout the article. In the sequel, let $m\in\mathbb{N}$ denote the multilinear parameter.

Let $\vec{p}=(p_1,p_2,\dots,p_m)$ be a vector of exponents such that  $1\leq p_i\leq \infty$ for every $i$. Let $\beta$, $\delta$ and $\tilde\delta$ be real constants.  Given $w$, $\vec{v}=(v_1,v_2,\dots,v_m)$ and $\vec{p}$, we say that $(w,\vec{v})\in \mathcal{H}_m(\vec{p}, \beta,\tilde\delta)$ if there exists a positive constant $C$ such that the inequality
	\begin{equation*}
    \frac{|B|^{1+(\delta-\tilde\delta)/n}}{w^{-1}(B)}\prod_{i=1}^m\left(\int_{\mathbb{R}^n}\frac{v_i^{-p_i'}}{(|B|^{1/n}+|x_B-y|)^{(n-\beta_i+\delta/m)p_i'}}\,dy\right)^{1/p_i'}\leq C
	\end{equation*}
	holds for every ball $B=B(x_B,R)$, with the  obvious changes when $p_i=1$. The numbers $\beta_i$ satisfy $\sum_{i=1}^m \beta_i=\beta$ and $0<\beta_i<n$, for every $i$. In Section~\ref{seccion: clase de pesos} we shall exhibit further details and basic properties of these classes of weights.

As we said above, the main topic of  this article is the study of the boundedness properties for higher order commutators related to some generalizations of the multilinear fractional integral operator of order $m$ defined as follows
	\[I_\alpha^m \vec{f}(x)=\int_{(\mathbb{R}^n)^m} \frac{\prod_{i=1}^m f_i(y_i)}{(\sum_{i=1}^m|x-y_i|)^{mn-\alpha}}\,d\vec{y},\]
	where $0<\alpha<mn$, $\vec{f}=(f_1,f_2,\dots, f_m)$ and $\vec{y}=(y_1,y_2,\dots, y_m)$. These operators are given by
	\[T_\alpha^m\vec{f}(x)=\int_{(\mathbb{R}^n)^m} K_\alpha(x,\vec{y})\prod_{i=1}^m f_i(y_i)\,d\vec{y},\]
	where $K_\alpha$ is a kernel satisfying some  standard size  and regularity conditions (see Section~\ref{seccion: preliminares}). 
 
We shall consider two different types of higher order commutators of $T_\alpha^m$.
	Given an $m$-tuple of functions $\bm{b}=(b_1,\dots,b_m)$, where  $b_i\in L^1_{\rm{loc}}$ for every $i$, the first type of multilinear commutator is given by a sum of linear  commutators as follows 
\[T_{\alpha,\bm{b}}^m\vec{f}(x)=\sum_{j=1}^m T_{\alpha,b_j}^m\vec{f}(x),\]
where formally 
\[T_{\alpha,b_j}^m\vec{f}(x)=\int_{(\mathbb{R}^n)^m}(b_j(x)-b_j(y_j))K_\alpha(x,\vec{y})\prod_{i=1}^m f_i(y_i)\,d\vec{y}\]
is the commutation of $b_j$ with $T_\alpha^m$ on the $j-$th entry of $T_\alpha^m$.

On the other hand we shall also consider a product type commutator which involves a recursive definition (see Section~\ref{seccion: preliminares}) and that can be expressed as
\[\mathcal{T}_{\alpha,\bm{b}}^m\vec{f}(x)=\int_{(\mathbb{R}^n)^m}K_\alpha(x,\vec{y})\prod_{i=1}^m(b_i(x)-b_i(y_i))f_i(y_i)\,d\vec{y}.\]
Both integral representations above are a consequence of the definition of the corresponding commutators and the representation for the operator $T_\alpha^m$. These type of commutators
 were introduced in \cite{PT03} and \cite{PPTT}, respectively.

Regarding the multilinear symbol involved, we shall consider multilinear functions with components belonging  to the classical Lipschitz spaces $\Lambda(\delta)$. The definition of these classes can be found in Section~\ref{seccion: preliminares}.

We are now in a position to state our main results. We shall denote  
	$1/p=\sum_{i=1}^m1/p_i$ along the manuscript. Further details can also be found in Section~\ref{seccion: preliminares}. 

\begin{teo}\label{teo: acotacion Lp Lipschitz para T_alpha,b (suma)}
Let $0<\alpha<mn$ and $T_\alpha^m$ be a multilinear fractional operator with kernel $K_\alpha$ satisfying \eqref{eq: condicion de tamaño} and \eqref{eq: condicion de suavidad}. Let $0<\delta<\min\{\gamma,mn-\alpha\}$, $\tilde\alpha=\alpha+\delta$ and $\vec{p}$ a vector of exponents that satisfies $p>n/\tilde\alpha$. Let  $\bm{b}=(b_1,\dots,b_m)$ a multilinear symbol such that $b_i\in\Lambda(\delta)$, for $1\leq i \leq m$. Let $\tilde\delta\leq \delta$ and  $(w,\vec{v})$ be a pair of weights in $\mathcal{H}_m(\vec{p},\tilde\alpha,\tilde\delta)$ such that $v_i^{-p_i'}\in\mathrm{RH}_m$, for every $i$ such that $1<p_i\leq \infty$. Then the  multilinear commutator $T_{\alpha,\bm{b}}^m$ is bounded from $\prod_{i=1}^mL^{p_i}(v_i^{p_i})$ to $\mathcal{L}_w(\tilde\delta)$, that is, there exists a positive constant $C$ such that the estimate
\[\frac{1}{w^{-1}(B)|B|^{\tilde\delta/n}}\int_B |T_{\alpha,\bm{b}}^m\vec{f}(x)-(T_{\alpha,\bm{b}}^m\vec{f})_B|\,dx\leq C\prod_{i=1}^m\|f_iv_i\|_{p_i}\]
holds for every ball $B$ and every $\vec{f}$ such that $f_iv_i\in L^{p_i}$, for $1\leq i\leq m$.
\end{teo}

The theorem above is a generalization to the multilinear setting of the estimate for commutators given in \cite{PR}. 

Concerning the commutator $\mathcal{T}_{\alpha,\bm{b}}^m$ we have the following result.

\begin{teo}\label{teo: acotacion Lp Lipschitz para T_alpha,b (producto)}
Let $0<\alpha<mn$ and $T_\alpha^m$ be a multilinear fractional operator with kernel $K_\alpha$ satisfying \eqref{eq: condicion de tamaño} and \eqref{eq: condicion de suavidad}. Let  $0<\delta<\min\{\gamma,(mn-\alpha)/m\}$, $\tilde\alpha=\alpha+m\delta$ and $\vec{p}$ a vector of exponents that satisfies $p>n/\tilde\alpha$. Let  $\bm{b}=(b_1,\dots,b_m)$ be a vector of symbols such that $b_i\in\Lambda(\delta)$, for $1\leq i \leq m$. Let $\tilde\delta\leq \delta$ and  $(w,\vec{v})$ be a pair of weights in $\mathcal{H}_m(\vec{p},\tilde\alpha,\tilde\delta)$ such that $v_i^{-p_i'}\in\mathrm{RH}_m$, for every $1<p_i\leq \infty$. Then the  multilinear commutator $\mathcal{T}_{\alpha,\bm{b}}^m$ is bounded from $\prod_{i=1}^mL^{p_i}(v_i^{p_i})$ to $\mathcal{L}_w(\tilde\delta)$, that is, there exists a positive constant $C$ such that the inequality
\[\frac{1}{w^{-1}(B)|B|^{\tilde\delta/n}}\int_B |\mathcal{T}_{\alpha,\bm{b}}^m\vec{f}(x)-(\mathcal{T}_{\alpha,\bm{b}}^m\vec{f})_B|\,dx\leq C\prod_{i=1}^m\|f_iv_i\|_{p_i}\]
holds for every ball $B$ and every $\vec{f}$ such that $f_iv_i\in L^{p_i}$, for $1\leq i\leq m$.
\end{teo}

The theorems above can be seen as the corresponding extensions of Theorem 1.1 and Theorem 1.2 in \cite{BPRe} since $\mathcal{H}_m$ and $\mathcal{L}_w(\delta)$ classes are wider than $\mathbb{H}_m$ and $\mathbb{L}_w(\delta)$, respectively. Note that the restriction on the parameter $\delta$ implicated is different in each theorem due to the nature of the considered commutators. 

Observe that the condition $v_i^{-p_i'}\in \mathrm{RH}_m$ plays a bump-type role, which is required when we deal with two weight estimates. We also observe that when $m=1$ this condition implies no additional assumption and the linear results are covered.   

\medskip

When we deal with one weight estimates we have the relation $w=\prod_{i=1}^m v_i$ and in this case  we say that $\vec{v}\in \mathcal{H}_m(\vec{p},\beta,\tilde\delta)$ to indicate that  $(w,\vec{v})$ is a pair in the class.
In \cite{BPR22(2)} we showed that when we restrict the pair of weights $(w,\vec{v})$ to satisfy this relation, the corresponding class of weights imply a reverse Hölder condition for $\prod_{i=1}^m v_i^{-1}$. We shall prove that this is also the case for the considered class  $\mathcal{H}_m(\vec{p},\beta,\tilde\delta)$.

The following result establishes the relation on the parameters required to have related weights on our class. 

\begin{teo}\label{teo: consecuencia de pesos iguales}
Let $0<\beta<mn$, $\tilde\delta\in\mathbb{R}$ and $\vec{p}$ be a vector of exponents. If $\vec{v}\in \mathcal{H}_m(\vec{p},\beta,\tilde\delta)$ and $\xi=p/(mp-1)>1$, then we have that $\tilde\delta=\beta-n/p$.
\end{teo}

The theorem above establishes that the pairs of related weights must occur in the border line $\tilde\delta=\beta-n/p$ of the regions depicted on page~\pageref{pag: regiones de pesos no triviales}. Under this assumption,  a consequence of this result is that the weight $w^{-1}=\prod_{i=1}^mv_i^{-1}$ belongs to $\text{RH}_\xi$, provided that $\xi=p/(mp-1)>1$, as states the following corollary.

\begin{coro}
	Let $0<\beta<mn$,  $\vec{p}$ a vector of exponents and $\tilde\delta=\beta-n/p$. If $\vec{v}\in \mathcal{H}_m(\vec{p},\beta,\tilde\delta)$ and $\xi=~p/(mp-1)>1$, then we have that $\prod_{i=1}^mv_i^{-1}\in \mathrm{RH}_\xi$.
\end{coro}

When $m=1$, this result establishes that $v^{-1}\in \mathrm{RH}_{p'}$, which was  proved in \cite{HSV}.

The remainder of the article is organized as follows. In Section~\ref{seccion: preliminares} we give the main definitions required in the sequel. The proof of our main results are contained in Section~\ref{seccion: conmutador suma} and~\ref{seccion: conmutador producto}, respectively. Finally, in Section~\ref{seccion: clase de pesos} we prove some properties of the class $\mathcal{H}_m(\vec{p},\beta,\tilde\delta)$ and exhibit the optimality of the associated parameters. 

\section{Preliminaries and definitions}\label{seccion: preliminares}

In the sequel $C$ shall denote an absolute positive constant that may change in every occurrence. By $A\lesssim B$ we mean that $A\leq C B$.  We say that $A\approx B$ when $A\lesssim B$ and $B\lesssim A$. 

By $m\in \mathbb{N}$ we denote the multilinear parameter involved in our estimates. Given a set $E$, $E^m$ shall denote the cartesian product of $E$ $m$ times. We shall be dealing with operators given by the expression
\begin{equation}\label{eq: operador T_alpha^m}
    T_{\alpha}^m\vec{f}(x)=\int_{(\mathbb{R}^n)^m}K_\alpha(x,\vec{y})\prod_{i=1}^mf_i(y_i)\,d\vec{y},
\end{equation}
for $0<\alpha<mn$, 
where $\vec{f}=(f_1,\dots,f_m)$, $\vec{y}=(y_1,\dots,y_m)$ and $K_\alpha$ is a kernel satisfying the size condition 
\begin{equation}\label{eq: condicion de tamaño}
    |K_\alpha(x,\vec{y})|\lesssim \frac{1}{(\sum_{i=1}^m|x-y_i|)^{mn-\alpha}}
\end{equation}
and an additional smoothness condition given by
\begin{equation}\label{eq: condicion de suavidad}
    |K_\alpha(x,\vec{y})-K_\alpha(x',\vec{y})|\lesssim \frac{|x-x'|^\gamma}{(\sum_{i=1}^m|x-y_i|)^{mn-\alpha+\gamma}},
\end{equation}
for some $0<\gamma\leq 1$, whenever $\sum_{i=1}^m|x-y_i|>c|x-x'|$, where $c$ is a fixed positive constant. It is easy to verify that the kernel $K_\alpha(x,\vec{y})=(\sum_{i=1}^m |x-y_i|)^{\alpha-mn}$ satisfy both conditions \eqref{eq: condicion de tamaño} and \eqref{eq: condicion de suavidad}. In this particular case we have $T_\alpha^m=I_\alpha^m$. 

Let us now introduce the two versions of commutators for the operators given above. Recall that if $T$ is a linear operator and  $b\in L^{1}_{\mathrm{loc}}$, the classical commutator $T_b$ or $[b,T]$ is formally given by the expression
\[[b,T]f=bTf-T(bf).\]
When we work on the multilinear setting, it is necessary to specify how we perform the commutation. If $b\in L^{1}_{\mathrm{loc}}$, $T$ is a multilinear operator and $\vec{f}=(f_1,f_2\dots,f_m)$ we write 
\[[b,T]_j(\vec{f})=bT(\vec{f})-T((f_1,\dots,bf_j,\dots,f_m)),\]
that is, $[b,T]_j$ is obtained by commuting $b$ with the $j$-th entry of $\vec{f}$. 

The first version of the commutator is defined as follows. Given an $m$-tuple $\bm{b}=(b_1,\dots,b_m)$, with $b_i\in  L^{1}_{\mathrm{loc}}$ for every $i$, we define the multilinear commutator of $T_{\alpha}^m$ by the expression
\[T_{\alpha,\bm{b}}^m\vec{f}(x)=\sum_{j=1}^m T_{\alpha,b_j}^m\vec{f}(x),\]
where
\[T_{\alpha,b_j}^m\vec{f}(x)=[b_j, T_{\alpha}^m]_{j}\vec f(x).\]
 
As a consequence of \eqref{eq: operador T_alpha^m} it is not difficult to see that
\begin{equation}\label{eq: representacion integral de T_a,bj^m}
T_{\alpha,b_j}^m\vec{f}(x)=\int_{(\mathbb{R}^n)^m}(b_j(x)-b_j(y_j))K_\alpha(x,\vec{y})\prod_{i=1}^m f_i(y_i)\,d\vec{y}.
\end{equation}

We now introduce the second type of commutator of $T_{\alpha}^m$. The
 multilinear product commutator  $\mathcal{T}_{\alpha,\bm{b}}^m$ is defined iteratively as follows
 \[\mathcal{T}_{\alpha,\bm{b}}^m\vec{f} = [b_m,\dots [b_{2},[b_1, T_\alpha^m]_1]_{2}\dots ]_m\vec{f}.\]
 It is difficult to use the expression above in practice. The next proposition, proved in \cite{BPRe}, gives a pointwise representation for the commutator by means of the operator $T_\alpha^m$.
Before stating the proposition, we shall include a notation that will be useful in the sequel.

 Let $S_m=\{0,1\}^m$. Given a set $B$ and $\sigma=(\sigma_1,\sigma_2,\dots,\sigma_m)\in S_m$, we define 
	\[B^{\sigma_i}=\left\{
	\begin{array}{ccl}
	B,&\textrm{ if }&\sigma_i=1\\
	\mathbb{R}^n\backslash B,&\textrm{ if }&\sigma_i=0.
	\end{array}
	\right.\]
	With the notation $\bm{B}^\sigma$ we will understand the cartesian product $B^{\sigma_1}\times B^{\sigma_2}\times\dots\times B^{\sigma_m}$. 
	
	Given $\sigma=(\sigma_1,\sigma_2,\dots,\sigma_m)$ we also define
	\[\bar{\sigma}_i=\left\{\begin{array}{ccl}
	   1  & \mathrm{ if } & \sigma_i=0;  \\
	   0  & \mathrm{ if } & \sigma_i=1,
	\end{array}
	\right.
	\]
	for every $1\leq i\leq m$ and $|\sigma|=\sum_{i=1}^m\sigma_i$. We shall write $\sigma=\bm{1}$ when $\sigma_i=1$ for every $i$ and $\sigma=\bm{0}$ when every $\sigma_i=0$. 

\begin{propo}\label{Integral Representation}
Let $T_\alpha^m$ be a multilinear operator as in \eqref{eq: operador T_alpha^m} and $\bm{b}=(b_1,\dots,b_m)$ where $b_i\in L^1_{\mathrm{loc}}$ for $1\leq i\leq m$. Then we have that
\[\mathcal{T}_{\alpha,\bm{b}}^m\vec{f}(x)=\sum_{\sigma\in S_m} (-1)^{m-|\sigma|}\left(\prod_{i=1}^m b_i^{\sigma_i}(x)\right)T_{\alpha}^m(f_1b_1^{\bar\sigma_1},\dots,f_m b_m^{\bar\sigma_m})(x).\]
Furthermore, we have the integral representation
\[\mathcal{T}_{\alpha,\bm{b}}^m\vec{f}(x)=\int_{(\mathbb{R}^n)^m}K_\alpha(x,\vec{y})\prod_{i=1}^m(b_i(x)-b_i(y_i))f_i(y_i)\,d\vec{y}.\]
\end{propo}
 
Regarding the symbols, we shall be dealing with the $\Lambda(\delta)$ Lipschitz spaces given, for $0<\delta<1$, by the collection of functions $b$ verifying 
\[|b(x)-b(y)|\le C|x-y|^{\delta}.\]
We denote by $\|b\|_{\Lambda(\delta)}$ the smallest constant for the inequality above to hold. If $\bm{b}=(b_1,\dots,b_m)$, with $b_i\in \Lambda(\delta)$ for every $1\leq i\leq m$, we define
	\[\|\bm{b}\|_{(\Lambda(\delta))^m}=\max_{1\leq i\leq m}\|b_i\|_{\Lambda(\delta)}.\]

We shall now turn to the definition of the classes of weights related to the main results. By a weight $w$ we understand a locally integrable function such that $0<w(x)<\infty$ for almost every $x$.
	 Let $\delta, \beta$ and $\tilde\delta$ be fixed real constants.  Given $w$, $\vec{v}=(v_1,v_2,\dots,v_m)$ and a vector of exponents $\vec{p}$, we say that $(w,\vec{v})\in \mathcal{H}_m(\vec{p}, \beta,\tilde\delta)$ if there exists a positive constant $C$ such that the inequality
	\begin{equation}\label{eq: condicion H_m(p,beta,delta tilde)-cal}
	    \frac{|B|^{1+(\delta-\tilde\delta)/n}}{w^{-1}(B)}\prod_{i=1}^m\left\|\frac{v_i^{-1}}{(|B|^{1/n}+|x_B-\cdot|)^{(n-\beta_i+\delta/m)}}\right\|_{p_i'}\leq C
	\end{equation}
	holds for every ball $B=B(x_B,R)$. We have $0<\beta_i<n$ for every $i$, $\sum_{i=1}^m \beta_i=\beta$ and this implies that $0<\beta<mn$. 
 
 Although the parameters $\delta$, $\tilde\delta$ and $\beta$ are independent to each other on this definition, they involve a relation that describes an area in which we can find nontrivial examples of pairs for this class (see Section~\ref{seccion: clase de pesos}).
 
We shall deal with different consequences of the inequality above throughout the article. We can derive them from the following one. Fix $\sigma\in S_m$, $\lambda>0$ and a ball $B=B(x_B,R)$. We shall bound every factor in the following way: when $\sigma_i=0$ we compute the norm in $\mathbb{R}^n\backslash \lambda B$ and for $\sigma_i=1$ we perform it in $\lambda B$. We get that condition \eqref{eq: condicion H_m(p,beta,delta tilde)-cal} implies
\begin{equation}\label{eq: condicion mezclada general}
\frac{|B|^{1+(\delta-\tilde\delta)/n-\theta(\sigma)}}{w^{-1}(B)}\prod_{i: \sigma_i=1} \left\|v_i^{-1}\mathcal{X}_{\lambda B}\right\|_{p_i'}\prod_{i: \sigma_i=0}\left\|\frac{v_i^{-1}\mathcal{X}_{\mathbb{R}^n\backslash \lambda B}}{|x_B-\cdot|^{n-\beta_i+\delta/m}}\right\|_{p_i'}\leq C_\lambda,
\end{equation}
where $\theta(\sigma)=\sum_{i: \sigma_i=1}(1-\beta_i/n+\delta/(mn))$. 

 From now on it will be useful to separate the coordinates of $\vec{p}$ in two sets. We define the sets $\mathcal{I}_1=\{1\leq i\leq m: p_i=1\}$ and $\mathcal{I}_2=\{1\leq i\leq m: 1<p_i\leq \infty\}$. By choosing $\sigma=\bm{1}$ and $\lambda=1$ in \eqref{eq: condicion mezclada general} we obtain
	\begin{equation}\label{eq: condicion local}
	     \frac{|B|^{1-\tilde\delta/n+\beta/n-1/p}}{w^{-1}(B)}\prod_{i\in\mathcal{I}_1}\|v_i^{-1}\mathcal{X}_B\|_\infty\prod_{i\in\mathcal{I}_2}\left(\frac{1}{|B|}\int_B v_i^{-p_i'}\right)^{1/p_i'}\leq C
	\end{equation}
 whilst if we take $\sigma=\bm{0}$ and $\lambda=1$, \eqref{eq: condicion mezclada general} becomes
	\begin{equation}\label{eq: condicion global}
	    \frac{|B|^{1+(\delta-\tilde\delta)/n}}{w^{-1}(B)}\prod_{i\in\mathcal{I}_1}\left\|\frac{v_i^{-1}\mathcal{X}_{\mathbb{R}^n\backslash B}}{|x_B-\cdot|^{n-\beta_i+\delta/m}}\right\|_\infty \,\prod_{i\in\mathcal{I}_2}\left(\int_{\mathbb{R}^n\backslash B}\frac{v_i^{-p_i'}}{|x_B-\cdot|^{(n-\beta_i+\delta/m)p_i'}}\right)^{1/p_i'}\leq C,
	\end{equation}
 for every ball $B$.
	We shall refer to these two last inequalities as the local and global conditions for $\mathcal{H}_m(\vec{p},\beta,\tilde\delta)$, respectively. 

We finish this section with the definition of reverse Hölder classes. Recall that a weight $w$ belongs to $\mathrm{RH}_s$, $1<s<\infty$, if there exists a positive constant $C$ such that the inequality
\[\left(\frac{1}{|B|}\int_B w^s\right)^{1/s}\leq \frac{C}{|B|}\int_B w\]
holds for every ball $B$ in $\mathbb{R}^n$.  The $\mathrm{RH}_{\infty}$ class is defined by the collection of weights $w$ for which  
\[\sup_B w\le  \frac{C}{|B|}\int_B w,\]
for some positive constant $C$. It is easy to see that $\mathrm{RH}_\infty\subset \mathrm{RH}_t\subset \mathrm{RH}_s$ whenever $1<s<t$. 	

The smallest constant for which the corresponding inequalities above hold is usually denoted by $[w]_{\mathrm{RH}_s}$, $1< s\leq \infty$.

\begin{obs}\label{obs: potencias positivas en RH inf}
It is not difficult to check that $w(x)=|x|^{\alpha}$ is an $\mathrm{RH}_\infty$ weight provided $\alpha\geq 0$. This property will be useful on the forthcoming estimates. 
\end{obs}

\section{Proof of Theorem~\ref{teo: acotacion Lp Lipschitz para T_alpha,b (suma)}}\label{seccion: conmutador suma}

We shall need a technical lemma that will be useful to deal with a local estimate on the proof of Theorem~\ref{teo: acotacion Lp Lipschitz para T_alpha,b (suma)}. A similar version of this result with different parameters was already proved in \cite{BPR22(2)} (see also \cite{BPR22}). We include the complete proof for the sake of completeness.

\begin{lema}\label{lema: acotacion local de Ialpha,m}
Let $0<\alpha<mn$, $0<\delta<mn-\alpha$, $\tilde\alpha=\alpha+\delta$ and $\vec{p}$ such that $p>n/\tilde\alpha$. Let $\tilde\delta\leq \delta$ and  $(w,\vec{v})$ be a pair of weights belonging  to the class $\mathcal{H}_m(\vec{p},\tilde\alpha,\tilde\delta)$ such that $v_i^{-p_i'}\in\mathrm{RH}_m$ for every $i\in\mathcal{I}_2$. Then there exists $C>0$ such that for every ball $B$ and every $\vec{f}$ such that $f_iv_i\in L^{p_i}$, $1\leq i\leq m$, we have that
\[\int_B |I_{\tilde\alpha}^m\vec{g}(x)|\,dx\leq C|B|^{\tilde\delta/n}w^{-1}(B)\prod_{i=1}^m\|f_iv_i\|_{p_i},\]
where $\vec{g}=\mathcal{X}_{2B}\vec{f}=(\mathcal{X}_{2B}f_1,...,\mathcal{X}_{2B}f_m).$ 
\end{lema}

\begin{proof}
We begin by splitting  $\mathcal I_2$ into the sets $\mathcal I_2^1$ and $\mathcal I_2^2$, where
		\[\mathcal{I}_2^1=\{i\in \mathcal{I}_2 : p_i<\infty\}\quad \textrm{ and } \quad\mathcal{I}_2^2=\{i \in \mathcal{I}_2: p_i=\infty\}.\]
		Define $m_1, m_2, m_2^1$ and $m_2^2$ as the cardinal of the sets $\mathcal{I}_1$, $\mathcal{I}_2$, $\mathcal{I}_2^1$ and $\mathcal{I}_2^2$, respectively. Then $m=m_1+m_2=m_1+m_2^1+m_2^2$. If $\tilde B=2B$, for $x\in B$ we have that
		\begin{align*}
		|I_{\tilde\alpha}^m\vec{g}(x)|&\leq \int_{\tilde B^m }\frac{\prod_{i=1}^m|f_i(y_i)|}{(\sum_{i=1}^m|x-y_i|)^{mn-\tilde\alpha}}\,d\vec{y}\\
		&\leq \left(\prod_{i\in\mathcal{I}_2^2}\|f_iv_i\|_\infty\right)\int_{\tilde B^m}\frac{\prod_{i\in \mathcal{I}_1\cup\mathcal{I}_2^1}|f_i(y_i)|\prod_{i\in \mathcal{I}_2^2}v_i^{-1}(y_i)}{(\sum_{i=1}^m|x-y_i|)^{mn-\tilde\alpha}}\,d\vec{y}\\
		&\leq \left(\prod_{i\in\mathcal{I}_2^2}\|f_iv_i\|_\infty\right)\left(\prod_{i\in \mathcal{I}_1}\|f_i\mathcal{X}_{\tilde B}\|_1\right) \int_{\tilde B^{m_2}}\frac{\prod_{i\in \mathcal{I}_2^1}|f_i(y_i)|\prod_{i\in \mathcal{I}_2^2}v_i^{-1}(y_i)}{(\sum_{i\in \mathcal{I}_2}|x-y_i|)^{mn-\tilde\alpha}}\,d\vec{y}\\
        &=\left(\prod_{i\in\mathcal{I}_2^2}\|f_iv_i\|_\infty\right)\left(\prod_{i\in \mathcal{I}_1}\|f_i\mathcal{X}_{\tilde B}\|_1\right)I(x,B).
		\end{align*}
		Since $p>n/\tilde\alpha$ we have that
		\[\tilde\alpha>n\sum_{i=1}^m\frac{1}{p_i}=m_1n+\frac{n}{p^*},\]
		where $1/p^*=\sum_{i\in\mathcal{I}_2}1/p_i$. Therefore we can decompose $\tilde\alpha=\tilde\alpha^1+\tilde\alpha^2$, such that  $\tilde\alpha^1>m_1n$ and $\tilde\alpha^2>n/p^*$. Then we get
		\[mn-\tilde\alpha=m_2n-\tilde\alpha^2+m_1n-\tilde\alpha^1.\]
		Let us sort the sets $\mathcal{I}_2^1$ and $\mathcal{I}_2^2$ increasingly, so
		\[\mathcal{I}_2^1=\left\{i_1,i_2,\dots,i_{m_2^1}\right\} \quad \textrm{ and } \quad \mathcal{I}_2^2=\left\{i_{m_2^1+1},i_{m_2^1+2},\dots,i_{m_2}\right\},\]
		and define $\vec{h}=(h_1,\dots,h_{m_2})$, where
		\[h_j=\left\{\begin{array}{ccl}
		|f_{i_j}|&\textrm{ if }&1\leq j\leq m_2^1;\\
		v_{i_j}^{-1}&\textrm{ if }&m_2^1+1\leq j\leq m_2.
		\end{array}
		\right.
		\]
		Then we can proceed in the following way
		\begin{align*}
		I(x,B)&\lesssim \int_{\tilde B^{m_2}}\frac{\prod_{i\in \mathcal{I}_2^1}|f_i(y_i)|\prod_{i\in \mathcal{I}_2^2}v_i^{-1}(y_i){(\sum_{i\in \mathcal{I}_2}|x-y_i|)^{{\tilde\alpha^1-m_1n}}}}{(\sum_{i\in \mathcal{I}_2}|x-y_i|)^{m_2n-\tilde\alpha^2}}\,d\vec{y}\\
		&\lesssim |\tilde B|^{\tilde\alpha^1/n-m_1}\int_{\tilde B^{m_2}}\frac{\prod_{i\in \mathcal{I}_2^1}|f_i(y_i)|\prod_{i\in \mathcal{I}_2^2}v_i^{-1}(y_i)}{(\sum_{i\in \mathcal{I}_2}|x-y_i|)^{m_2n-\tilde\alpha^2}}\,d\vec{y}\\
		&=|\tilde B|^{\tilde\alpha^1/n-m_1}\int_{\tilde B^{m_2}}\frac{\prod_{j=1}^{m_2}h_j(y_{i_j})}{(\sum_{j=1}^{m_2}|x-y_{i_j}|)^{m_2n-\tilde\alpha^2}}\,d\vec{y}.
		\end{align*}
		Next we define the vector of exponents $\vec{r}=(r_1,\dots,r_{m_2})$ in the following way
		\[r_j=\left\{\begin{array}{ccr}
		m_2p_{i_j}/(m_2-1+p_{i_j})&\textrm{ if }&1\leq j\leq m_2^1;\\
		m_2&\textrm{ if }&m_2^1+1\leq j\leq m_2.
		\end{array}
		\right.
		\]
		With this definition we obtain
		\begin{align*}
		\frac{1}{r}&=\sum_{j=1}^{m_2}\frac{1}{r_j}=\sum_{j=1}^{m_2^1}\left(\frac{1}{m_2}+\frac{m_2-1}{m_2p_{i_j}}\right)+\sum_{j=m_2^1+1}^{m_2}\frac{1}{m_2}=\frac{m_2^1}{m_2}+\frac{m_2-1}{m_2{p^*}}+\frac{m_2^2}{m_2}=1+\frac{m_2-1}{m_2p^*}.
		\end{align*}
		
		Notice that $1/r>1/p^*$ and  also  $n/p^*<\tilde\alpha^2$ by construction. Then there exists a number $\tilde\alpha_0$ such that $n/p^*<\tilde\alpha_0<n/r$. Indeed, if $\tilde\alpha^2<n/r$ we can directly pick $\tilde\alpha_0=\tilde\alpha^2$. Otherwise $\tilde\alpha_0<\tilde\alpha^2$. Let us first consider the case $m_2\geq 2$. We set
		\[\frac{1}{q}=\frac{1}{r}-\frac{\tilde\alpha_0}{n}.\]
		Then $0<1/q<1$ since
		\[\frac{1}{r}-1=\left(1-\frac{1}{m_2}\right)\frac{1}{p^*}<\frac{1}{p^*}<\frac{\tilde\alpha_0}{n}.\]
				By using the well-known continuity property $I_{\tilde\alpha_0}^{m_2}\colon \prod_{j=1}^{m_2} L^{r_j}\to L^q$ with respect to the Lebesgue measure (see, for example, \cite{Moen09}) we get
		\begin{align*}
		\int_B I(x,B)\,dx&\lesssim |\tilde B|^{\tilde\alpha^1/n-m_1+(\tilde\alpha^2-\tilde\alpha_0)/n}\left(\int_B|I_{\tilde\alpha_0}^{m_2}(\vec{h}\mathcal{X}_{\tilde B^{m_2}})(x)|^q\,dx\right)^{1/q}|B|^{1/q'}\\
		&\lesssim |\tilde B|^{(\tilde\alpha-\tilde\alpha_0)/n-m_1+1/q'}\left(\int_{\mathbb{R}^n}|I_{\tilde\alpha_0}^{m_2}(\vec{h}\mathcal{X}_{\tilde B^{m_2}})(x)|^q\,dx\right)^{1/q}\\
		&\lesssim |\tilde B|^{(\tilde\alpha-\tilde\alpha_0)/n-m_1+1/q'}\prod_{j=1}^{m_2}\|h_j\mathcal{X}_{\tilde B}\|_{r_j}.
		\end{align*}
		Observe that $r_j<p_{i_j}$ for every $1\leq j\leq m_2^1$. Since $v_i^{-p_i'}\in \mathrm{RH}_m\subset \mathrm{RH}_{m_2}$ for every $i\in\mathcal{I}_2$, by applying H\"{o}lder inequality and then the reverse Hölder condition on these weights we get
		\begin{align*}	\prod_{j=1}^{m_2}\|h_j\mathcal{X}_{\tilde B}\|_{r_j}&=\prod_{i\in\mathcal{I}_2^1}\left(\int_{\tilde B}|f_i|^{r_i}v_i^{r_i}v_i^{-r_i}\right)^{1/r_i}\prod_{i\in \mathcal{I}_2^2}\left(\int_{\tilde B}v_i^{-m_2}\right)^{1/m_2}\\
		&\leq \prod_{i\in\mathcal{I}_2^1}\|f_iv_i\|_{p_i}\left(\int_{\tilde B} v_i^{-m_2p_i'}\right)^{1/(m_2p_i')}\prod_{i\in \mathcal{I}_2^2}\left(\int_{\tilde B}v_i^{-m_2}\right)^{1/m_2}\\
		&\leq |\tilde B|^{m_2^1/m_2-1/(m_2p^*)+m_2^2/m_2}\prod_{i\in\mathcal{I}_2^1}\left[v_i^{-p_i'}\right]_{\mathrm{RH}_{m_2}}\|f_iv_i\|_{p_i}\left(\frac{1}{|\tilde B|}\int_{\tilde B} v_i^{-p_i'}\right)^{1/p_i'}\\
		&\quad \times \prod_{i\in \mathcal{I}_2^2}\left[v_i^{-1}\right]_{\mathrm{RH}_{m_2}}\left(\frac{1}{|\tilde B|}\int_{\tilde B}v_i^{-1}\right)\\
		&\lesssim |\tilde B|^{1-1/(m_2p^*)}\prod_{i\in\mathcal{I}_2^1}\|f_iv_i\|_{p_i}\left(\frac{1}{|\tilde B|}\int_{\tilde B} v_i^{-p_i'}\right)^{1/p_i'}\prod_{i\in \mathcal{I}_2^2}\left(\frac{1}{|\tilde B|}\int_{\tilde B}v_i^{-1}\right).
		\end{align*}
		By combining the estimates above, we finally arrive to
		\begin{align*}
		\int_B |I_{\tilde\alpha}^m\vec{g}(x)|\,dx&\leq \left(\prod_{i\in\mathcal{I}_2^2}\|f_iv_i\|_\infty\right)\left(\prod_{i\in \mathcal{I}_1}\|f_i\mathcal{X}_{\tilde B}\|_1\right)\int_B I(x,B)\,dx\\
		&\lesssim \left(\prod_{i\in\mathcal{I}_2^2}\|f_iv_i\|_\infty\right)\left(\prod_{i\in \mathcal{I}_1}\|f_i\mathcal{X}_{\tilde B}\|_1\right)|\tilde B|^{(\tilde\alpha-\tilde\alpha_0)/n-m_1+1/q'+1-1/(m_2p^*)}\\
		&\qquad \times \prod_{i\in\mathcal{I}_2^1}\|f_iv_i\|_{p_i}\left(\frac{1}{|\tilde B|}\int_{\tilde B} v_i^{-p_i'}\right)^{1/p_i'}\prod_{i\in \mathcal{I}_2^2}\left(\frac{1}{|\tilde B|}\int_{\tilde B}v_i^{-1}\right)\\
		&\lesssim \left(\prod_{i=1}^m \|f_iv_i\|_{p_i}\right)\prod_{i\in\mathcal{I}_1} \left\|v_i^{-1}\mathcal{X}_{\tilde B}\right\|_\infty\prod_{i\in \mathcal{I}_2}\left(\frac{1}{|\tilde B|}\int_{\tilde B} v_i^{-p_i'}\right)^{1/p_i'} \\
		&\qquad \times |\tilde B|^{(\tilde\alpha-\tilde\alpha_0)/n-m_1+1/q'+1-1/(m_2p^*)}.
		\end{align*}
  By applying condition \eqref{eq: condicion mezclada general} with $\lambda=2$ and $\sigma=\bm{1}$, we arrive to
  
\begin{align*}
\int_B |I_{\tilde\alpha}^m\vec{g}(x)|\,dx& \lesssim \left(\prod_{i=1}^m \|f_iv_i\|_{p_i}\right)|B|^{\tilde\delta/n+1/p-1-\tilde\alpha/n+(\tilde\alpha-\tilde\alpha_0)/n-m_1+1/q'+1-1/(m_2p^*)}w^{-1}(B)\\
&=\left(\prod_{i=1}^m \|f_iv_i\|_{p_i}\right)|B|^{\tilde\delta/n}w^{-1}(B),
  \end{align*} 
  so the proof is complete when $m_2\geq 2$. 
  
  In order to finish we shall provide the estimate for $0\leq m_2<2$. There are only three possible cases:  
		\begin{enumerate}
			\item $m_2=0$. In this case we have $m_2^1=m_2^2=0$ and this implies $\vec{p}=(1,1,\dots,1)$. This situation is not possible, because $p>n/\tilde\alpha$.
			\item $m_2^1=0$ and $m_2^2=1$. In this case $1/p=m-1$. Condition $p>n/\tilde\alpha$ implies $\tilde\alpha>(m-1)n$. Let $i_0$ be the index such that $p_{i_0}=\infty$. By using Fubini's theorem, we can proceed in the following way
			\[\int_B \int_{\tilde B^m}\frac{\prod_{i=1}^m|f_i(y_i)|}{(\sum_{i=1}^m|x-y_i|)^{mn-\tilde\alpha}}\,d\vec{y}\,dx=\int_{\tilde B^m}\prod_{i=1}^m|f_i(y_i)|\left(\int_B \left(\sum_{i=1}^m|x-y_i|\right)^{\tilde\alpha-mn}\,dx\right)\,d\vec{y}.\]
			Since
			\begin{align*}
			\int_B \left(\sum_{i=1}^m|x-y_i|\right)^{\tilde\alpha-mn}\,dx&\lesssim \int_0^{4R}\rho^{\tilde\alpha-mn}\rho^{n-1}\,d\rho\\
			&\lesssim  |B|^{\tilde\alpha/n-m+1},
			\end{align*}
			by applying Hölder inequality,  and \eqref{eq: condicion mezclada general} (again with $\lambda=2$ and $\sigma=\bm{1}$), we get
			\begin{align*}
			\int_B |I_{\tilde\alpha}^m\vec{g}(x)|\,dx&\lesssim |B|^{\tilde\alpha/n-m+2}\left(\prod_{i=1}^m\|f_iv_i\|_{p_i}\right)\left(\prod_{i\in \mathcal{I}_1}\left\|v_i^{-1}\mathcal{X}_{\tilde B}\right\|_\infty\right) \left(\frac{1}{|\tilde B|}\int_{\tilde B}v_{i_0}^{-1}\right)\\
			&\lesssim \left(\prod_{i=1}^m\|f_iv_i\|_{p_i}\right)| B|^{\tilde\alpha/n-m+2+\tilde\delta/n-\tilde\alpha/n+1/p-1}w^{-1}(B)\\
			&= \left(\prod_{i=1}^m\|f_iv_i\|_{p_i}\right)|B|^{\tilde\delta/n}w^{-1}(B).
			\end{align*}
			\item $m_2^1=1$ and $m_2^2=0$. If $i_0$ denotes the index for which $1<p_{i_0}<\infty$, the condition $p>n/\tilde\alpha$ implies that
			\[\frac{\tilde\alpha}{n}>\frac{1}{p}=m-1+\frac{1}{p_{i_0}},\]
			and thus $\tilde\alpha>(m-1)n$. By repeating the steps in the previous case we arrive to 
			\begin{align*}
			\int_B |I_{\tilde\alpha}^m\vec{g}(x)|\,dx&\lesssim |B|^{\tilde\alpha/n-m+1+1/p_{i_0}'}\left(\prod_{i=1}^m\|f_iv_i\|_{p_i}\right)\left(\prod_{i\in \mathcal{I}_1}\left\|v_i^{-1}\mathcal{X}_{\tilde B}\right\|_\infty\right)\left(\frac{1}{|\tilde B|}\int_{\tilde B}v_{i_0}^{-p_{i_0}'}\right)^{1/p_{i_0}'}\\
            &\lesssim \left(\prod_{i=1}^m\|f_iv_i\|_{p_i}\right)
            |B|^{\tilde\alpha/n-m+1+1/p_{i_0}'+\tilde\delta/n-\tilde\alpha/n+1/p-1}w^{-1}(B)\\
            &=\left(\prod_{i=1}^m\|f_iv_i\|_{p_i}\right)|B|^{\tilde\delta/n}w^{-1}(B).
			\end{align*}
		\end{enumerate}
		This covers all the possible cases for $m_2$ and the proof is complete.\qedhere
\end{proof}

\medskip

\begin{proof}[Proof of Theorem~\ref{teo: acotacion Lp Lipschitz para T_alpha,b (suma)}]
It is enough to prove the estimate for $T_{\alpha,b_j}^m\vec{f}$ with $C$ independent of $j$. Furthermore, we only need to show that
\begin{equation}\label{eq: teo: acotacion Lp Lipschitz para T_alpha,b (suma) - eq1}
\frac{1}{w^{-1}(B)|B|^{\tilde\delta/n}}\int_B \left|T_{\alpha,b_j}^{m}\vec{f}(x)-c_j\right|\,dx\leq C\prod_{i=1}^m\|f_iv_i\|_{p_i},
\end{equation}
for some positive constant $c_j$ and every ball $B$, with $C$ independent of $B$ and $j$. From \eqref{eq: teo: acotacion Lp Lipschitz para T_alpha,b (suma) - eq1}, by choosing $c=\sum_{j=1}^m c_j$, the result will immediately follow by a simple summation on $j$.

Fix $1\leq j\leq m$ and a ball $B=B(x_B, R)$. We decompose $\vec{f}=(f_1,f_2,\dots,f_m)$ as $\vec{f}=\vec{f}_1+\vec{f}_2$, where $\vec{f}_1=\mathcal{X}_{2B}\vec{f}$. We pick 
\begin{align*}
c_j=\left(T_{\alpha,b_j}^{m}\vec{f}_2\right)_B
&=\frac{1}{|B|}\int_B T_{\alpha,b_j}^{m}\vec{f}_2(z)\,dz\\
&=\sum_{\sigma\in S_m,\sigma\neq \bf{1}}\frac{1}{|B|}\int_B\int_{(\bm{2B})^\sigma} (b_j(z)-b_j(y_j))K_\alpha(z,\vec{y})\prod_{i=1}^mf_i(y_i)\,d\vec{y}\,dz,
\end{align*}
by virtue of \eqref{eq: representacion integral de T_a,bj^m}.

Then we proceed as follows
\begin{align*}
    \frac{1}{w^{-1}(B)|B|^{\tilde\delta/n}}\int_B |T_{\alpha,b_j}^{m}\vec{f}(x)-c_j|\,dx&\leq\frac{1}{w^{-1}(B)|B|^{\tilde\delta/n}}\left(\int_B |T_{\alpha,b_j}^{m}\vec{f}_1(x)|\,dx+\int_B|T_{\alpha,b_j}^{m}\vec{f}_2(x)-c_j|\,dx\right)\\
    &=\frac{1}{w^{-1}(B)|B|^{\tilde\delta/n}}\left(I_1+I_2\right).
\end{align*}
We begin with the estimate of $I_1$. By Lemma~\ref{lema: acotacion local de Ialpha,m} we obtain 
\begin{align*}
I_1=\int_B |T_{\alpha,b_j}^{m}\vec{f}_1(x)|\,dx&\leq\int_B\int_{(2B)^m}|b_j(x)-b_j(y_j)|\,|K_\alpha(x,\vec{y})|\prod_{i=1}^m|f_i(y_i)|\,d\vec{y}\,dx\\
&\lesssim\|b_j\|_{\Lambda(\delta)}\int_B\int_{(2B)^m}\frac{|x-y_j|^{\delta}}{(\sum_{i=1}^m|x-y_i|)^{mn-\alpha}}\prod_{i=1}^m|f_i(y_i)|\,d\vec{y}\,dx\\
&\lesssim\|\bm{b}\|_{(\Lambda(\delta))^m}\int_B |I_{\tilde\alpha}^m\vec{f}_1(x)|\,dx\\
&\lesssim \|\bm{b}\|_{(\Lambda(\delta))^m}|B|^{\tilde\delta/n}w^{-1}(B)\prod_{i=1}^m\|f_iv_i\|_{p_i}.
\end{align*}
This yields
\[\frac{I_1}{w^{-1}(B)|B|^{\tilde\delta/n}} \lesssim \|\bm{b}\|_{(\Lambda(\delta))^m}\prod_{i=1}^m\|f_iv_i\|_{p_i}.\]
Let us now estimate $I_2$. From our choice of $c_j$, we first observe that
\begin{align*}
\int_B |T_{\alpha, b_j}^m\vec{f}_2(x)-c_j|\,dx&\leq \frac{1}{|B|}\int_B\int_B|T_{\alpha,b_j}^m\vec{f}_2(x)-T_{\alpha,b_j}^m\vec{f}_2(z)|\,dz\,dx\\
&\hspace{-1truecm}\leq \frac{1}{|B|}\sum_{\sigma\in S_m,\sigma\neq \bf{1}}\int_B\int_B\int_{(\bm{2B})^\sigma}\left|(b_j(x)-b_j(y_j))K_\alpha(x,\vec{y})-(b_j(z)-b_j(y_j))K_\alpha(z,\vec{y})\right|\\
&\qquad \times\prod_{i=1}^m|f_i(y_i)|\,d\vec{y}\,dz\,dx.
\end{align*}
Therefore
\begin{align*}
I_2 &\leq \frac{1}{|B|}\sum_{\sigma\in S_m,\sigma\neq \bf{1}}\int_B\int_B\int_{(\bm{2B})^\sigma}\left|(b_j(x)-b_j(y_j))(K_\alpha(x,\vec{y})-K_\alpha(z,\vec{y}))\right|\prod_{i=1}^m|f_i(y_i)|\,d\vec{y}\,dz\,dx\,\,\\
&\quad + \frac{1}{|B|}\sum_{\sigma\in S_m,\sigma\neq \bf{1}}\int_B\int_B\int_{(\bm{2B})^\sigma}\left|(b_j(x)-b_j(z))K_\alpha(z,\vec{y})\right|\prod_{i=1}^m|f_i(y_i)|\,d\vec{y}\,dz\,dx\\
&= \frac{1}{|B|}\sum_{\sigma\in S_m,\sigma\neq \bf{1}} (I_2^1(\sigma)+I_2^2(\sigma)).
\end{align*} 
Let us estimate every term above separately. Fix $\sigma\in S_m,\sigma\neq \bf{1}$. Since $\sigma\neq \bf{1}$, there exists at least one index $i_0$ such that $y_{i_0}\not\in 2B$. Then
\[\sum_{i=1}^m |x-y_i|\geq |x-y_{i_0}|>R>\frac{|x-z|}{2}.\]
 So we can apply condition
\eqref{eq: condicion de suavidad} in order to get
\begin{align*}
  |K_\alpha(x,\vec{y})-K_\alpha(z,\vec{y})|&\lesssim \frac{|x-z|^\gamma}{(\sum_{i=1}^m |x-y_i|)^{mn-\alpha+\gamma}}\\
  &\lesssim \frac{|B|^{\gamma/n}}{(\sum_{i=1}^m |x-y_i|)^{mn-\alpha+\gamma}}.
\end{align*}
Recalling that $\delta<\gamma$, we can proceed with $I_2^1(\sigma)$ as follows 
\begin{align*}
I_2^1(\sigma)&\lesssim\|b_j\|_{\Lambda(\delta)}|B|^{\gamma/n}\int_B\int_B \int_{(\bm{2B})^\sigma}\frac{|x-y_j|^\delta\prod_{i=1}^m|f_i(y_i)|}{(\sum_{i=1}^m|x-y_i|)^{mn-\alpha+\gamma}}\,d\vec{y}\,dz\,dx\\
&\lesssim\|\bm{b}\|_{(\Lambda(\delta))^m} |B|^{1+\gamma/n}\int_B\int_{(\bm{2B})^\sigma}\frac{\prod_{i=1}^m|f_i(y_i)|}{(\sum_{i=1}^m|x-y_i|)^{mn-\tilde\alpha+\delta+\gamma-\delta}}\,d\vec{y}\,dx\\
&\lesssim \|\bm{b}\|_{(\Lambda(\delta))^m}|B|^{1+\delta/n}\int_B\int_{(\bm{2B})^\sigma}\frac{\prod_{i=1}^m|f_i(y_i)|}{(\sum_{i=1}^m|x_B-y_i|)^{mn-\tilde\alpha+\delta}}\,d\vec{y}\,dx\\
&=\|\bm{b}\|_{(\Lambda(\delta))^m}|B|^{2+\delta/n}F(x_B,\sigma).
\end{align*}
In order to estimate the term $F(x_B,\sigma)$, we conveniently separate it into a product of integrals, and combine Hölder inequality with condition \eqref{eq: condicion mezclada general}, applied with $\sigma$ and $\lambda=2$. We obtain \refstepcounter{BPR}\label{pag: estimacion de F(x_B,sigma)} 
\begin{align*}
F(x_B,\sigma)&\lesssim \left(\prod_{i:\sigma_i=1} \int_{2B}\frac{|f_i(y_i)|}{|2B|^{1-\tilde \alpha_i/n+\delta/(mn)}}\,dy_i\right)\left(\prod_{i: \sigma_i=0} \int_{\mathbb{R}^n\backslash 2B}\frac{|f_i(y_i)|}{|x_B-y_i|^{n-\tilde\alpha_i+\delta/m}}\,dy_i\right)\\
&\lesssim \prod_{i=1}^m\|f_iv_i\|_{p_i}\left(\prod_{i:\sigma_i=1} \left\|\frac{v_i^{-1}\mathcal{X}_{ 2B}}{|2B|^{1-\tilde\alpha_i/n+\delta/(mn)}}\right\|_{p_i'}\right)\left(\prod_{i:\sigma_i=0} \left\|\frac{v_i^{-1}\mathcal{X}_{\mathbb{R}^n\backslash 2B}}{|x_B-\cdot|^{n-\tilde\alpha_i+\delta/m}}\right\|_{p_i'}\right)\\
&=\left(\prod_{i=1}^m\|f_iv_i\|_{p_i}\right)|2B|^{-\sum_{i:\sigma_i=1}(1-\tilde\alpha_i/n+\delta/(mn))}\prod_{i: \sigma_i=1} \left\|v_i^{-1}\mathcal{X}_{2B}\right\|_{p_i'} \prod_{i:\sigma_i=0}\left\|\frac{v_i^{-1}\mathcal{X}_{\mathbb{R}^n\backslash 2B}}{|x_B-\cdot|^{n-\tilde\alpha_i+\delta/m}}\right\|_{p_i'}\\
&\lesssim \left(\prod_{i=1}^m\|f_iv_i\|_{p_i}\right) |B|^{-1+(\tilde\delta-\delta)/n}w^{-1}(B).
\end{align*}
Thus
\begin{equation}\label{eq: teo: acotacion Lp Lipschitz para T_alpha,b (suma) - eq3}
I_2^1(\sigma)\lesssim\|\bm{b}\|_{(\Lambda(\delta))^m}\left(\prod_{i=1}^m\|f_iv_i\|_{p_i}\right)|B|^{1+\tilde\delta/n}w^{-1}(B).
\end{equation}
It only remains to perform the estimate for $I_2^2(\sigma)$. By using the size condition for the kernel \eqref{eq: condicion de tamaño} we can write
\begin{align*}
I_2^2(\sigma)&\lesssim \|b_j\|_{\Lambda(\delta)}|B|^{\delta/n}\int_B\int_B\int_{(\bm{2B})^\sigma}\frac{\prod_{i=1}^m|f_i(y_i)|}{(\sum_{i=1}^m |z-y_i|)^{mn-\alpha}}\,d\vec{y}\,dx\,dz\\
&\lesssim \|\bm{b}\|_{(\Lambda(\delta))^m}|B|^{\delta/n}\int_B\int_B\int_{(\bm{2B})^\sigma}\frac{\prod_{i=1}^m|f_i(y_i)|}{(\sum_{i=1}^m |x_B-y_i|)^{mn-\alpha}}\,d\vec{y}\,dx\,dz\\
&= \|\bm{b}\|_{(\Lambda(\delta))^m}|B|^{2+\delta/n}F(x_B,\sigma)\\
&\lesssim \|\bm{b}\|_{(\Lambda(\delta))^m}\left(\prod_{i=1}^m \|f_iv_i\|_{p_i}\right)|B|^{1+\tilde\delta/n}w^{-1}(B),
\end{align*}
where we have used the estimate above for $F(x_B,\sigma)$. Consequently, we get
\[I_2\lesssim \|\bm{b}\|_{(\Lambda(\delta))^m}\left(\prod_{i=1}^m \|f_iv_i\|_{p_i}\right)|B|^{\tilde\delta/n}w^{-1}(B).\]
We obtained the desired bound for both $I_1$ and $I_2$. The proof is now complete.
\end{proof}

\section{Proof of Theorem~\ref{teo: acotacion Lp Lipschitz para T_alpha,b (producto)}}\label{seccion: conmutador producto}

We devote this section to prove Theorem~\ref{teo: acotacion Lp Lipschitz para T_alpha,b (producto)}. We shall first establish two auxiliary lemmas. The first one is essentially the boundedness given in Lemma~\ref{lema: acotacion local de Ialpha,m} with different parameters and the proof can be achieved by following the same steps. The second one, proved in \cite{BPRe}, is an estimate for differences of products.

\begin{lema}\label{lema: acotacion local de Ialpha,m (producto)}
Let $0<\alpha<mn$, $0<\delta<(mn-\alpha)/m$, $\tilde\alpha=\alpha+m\delta$ and $\vec{p}$ a vector of exponents that satisfies $p>n/\tilde\alpha$. Let $\tilde\delta\leq \delta$ and  $(w,\vec{v})$ be a pair of weights belonging  to the class  $\mathcal{H}_m(\vec{p},\tilde\alpha,\tilde\delta)$  such that $v_i^{-p_i'}\in\mathrm{RH}_m$ for every $i\in\mathcal{I}_2$. Then there exists a positive constant $C$ such that for every ball $B$ and every $\vec{f}$ such that $f_iv_i\in L^{p_i}$, $1\leq i\leq m$, we have that
\[\int_B |I_{\tilde\alpha}^m\vec{g}(x)|\,dx\lesssim|B|^{\tilde\delta/n}w^{-1}(B)\prod_{i=1}^m\|f_iv_i\|_{p_i},\]
where $\vec{g}=\mathcal{X}_{2B}\vec{f}$. 
\end{lema}

\begin{lema}\label{lema: estimacion diferencia de productos}
Let $m\in\mathbb{N}$ and $a_i,b_i$ and $c_i$ be real numbers for $1\leq i\leq m$. Then
\[\prod_{i=1}^m (a_i-b_i)-\prod_{i=1}^m (c_i-b_i)=\sum_{j=1}^m (a_j-c_j)\prod_{i<j}(a_i-b_i)\prod_{i>j}(c_i-b_i). \]
\end{lema}

\medskip

\begin{proof}[Proof of Theorem~\ref{teo: acotacion Lp Lipschitz para T_alpha,b (producto)}]
It will be enough to prove that
\begin{equation}\label{eq: teo: acotacion Lp Lipschitz para T_alpha,b (producto) - eq1}
\frac{1}{w^{-1}(B)|B|^{\tilde\delta/n}}\int_B |\mathcal{T}_{\alpha,\bm{b}}^{m}\vec{f}(x)-c|\,dx\lesssim\prod_{i=1}^m\|f_iv_i\|_{p_i},
\end{equation}
for some constant $c$ and every ball $B$, with $C$ independent of $B$ and $\vec{f}$. 

Fix a ball $B=B(x_B, R)$. By proceeding as in the proof of Theorem~\ref{teo: acotacion Lp Lipschitz para T_alpha,b (suma)}, we split  $\vec{f}=\vec{f}_1+\vec{f}_2$, where $\vec{f}_1=\mathcal{X}_{2B}\vec{f}$. We take 
\[c=\left(\mathcal{T}_{\alpha,\bm{b}}^{m}\vec{f}_2\right)_B=\frac{1}{|B|}\int_B \mathcal{T}_{\alpha,\bm{b}}^{m}\vec{f}_2(z)\,dz.\]
By Proposition~\ref{Integral Representation}, for $z\in B$ we have that
\begin{equation}\label{eq: teo: acotacion Lp Lipschitz para T_alpha,b (producto) - eq2}
 \mathcal{T}_{\alpha,\bm{b}}^{m}\vec{f}_2(z)=\sum_{\sigma\in S_m,\sigma\neq \bf{1}}\int_{(\bm{2B})^\sigma} K_\alpha(z,\vec{y})\prod_{i=1}^m(b_i(z)-b_i(y_i))f_i(y_i)\,d\vec{y}.
\end{equation}

Thus
\begin{align*}
    \frac{1}{w^{-1}(B)|B|^{\tilde\delta/n}}\int_B |\mathcal{T}_{\alpha,\bm{b}}^{m}\vec{f}(x)-c|\,dx&\leq\frac{1}{w^{-1}(B)|B|^{\tilde\delta/n}}\int_B |\mathcal{T}_{\alpha,\bm{b}}^{m}\vec{f}_1(x)|\,dx\\
    &\qquad +\frac{1}{w^{-1}(B)|B|^{1+\tilde\delta/n}}\int_B\int_B|\mathcal{T}_{\alpha,\bm{b}}^{m}\vec{f}_2(x)-\mathcal{T}_{\alpha,\bm{b}}^{m}\vec{f}_2(z)|\,dz\,dx\\
    &=\frac{1}{w^{-1}(B)|B|^{\tilde\delta/n}}\left(I+II\right).
\end{align*}
Let us first estimate $I$. By applying Proposition~\ref{Integral Representation}, \eqref{eq: condicion de tamaño} and Lemma~\ref{lema: acotacion local de Ialpha,m (producto)} we get 
\begin{align*}
I=\int_B |\mathcal{T}_{\alpha,\bm{b}}^{m}\vec{f}_1(x)|\,dx&\leq\int_B\int_{(2B)^m}\,|K_\alpha(x,\vec{y})|\prod_{i=1}^m|b_i(x)-b_i(y_i)|\,|f_i(y_i)|\,d\vec{y}\,dx\\
&\lesssim\prod_{i=1}^m\|b_i\|_{\Lambda(\delta)}\int_B\int_{(2B)^m}|K_{\tilde\alpha}(x,\vec{y})|\prod_{i=1}^m|f_i(y_i)|\,d\vec{y}\,dx\\
&\lesssim\|\bm{b}\|_{(\Lambda(\delta))^m}^m\int_B |I_{\tilde\alpha}^m\vec{f}_1(x)|\,dx\\
&\lesssim\|\bm{b}\|_{(\Lambda(\delta))^m}^m|B|^{\tilde\delta/n}w^{-1}(B)\prod_{i=1}^m\|f_iv_i\|_{p_i}.
\end{align*}
Consequently,
\[\frac{1}{w^{-1}(B)|B|^{\tilde\delta/n}}\,\,I\lesssim \|\bm{b}\|_{(\Lambda(\delta))^m}^m\prod_{i=1}^m\|f_iv_i\|_{p_i}.\]

We continue with the estimate of $II$. We shall see that
\begin{equation}\label{eq: teo: acotacion Lp Lipschitz para T_alpha,b (producto) - eq3}
|\mathcal{T}_{\alpha,\bm{b}}^{m}\vec{f}_2(x)-\mathcal{T}_{\alpha,\bm{b}}^{m}\vec{f}_2(z)|\lesssim\|\bm{b}\|_{(\Lambda(\delta))^m}^m|B|^{\tilde\delta/n-1}w^{-1}(B)\prod_{i=1}^m\|f_iv_i\|_{p_i},
\end{equation}
for every $x,z\in B$. This would imply that $II\lesssim\|\bm{b}\|_{(\Lambda(\delta))^m}^m|B|^{\tilde\delta/n}w^{-1}(B)\prod_{i=1}^m\|f_iv_i\|_{p_i}$ and the desired estimate would follow.

By \eqref{eq: teo: acotacion Lp Lipschitz para T_alpha,b (producto) - eq2}, for $x\in B$ we can write
\begin{align*}
    |\mathcal{T}_{\alpha,\bm{b}}^{m}&\vec{f}_2(x)-\mathcal{T}_{\alpha,\bm{b}}^{m}\vec{f}_2(z)|\\
    &\leq \sum_{\sigma\in S_m,\sigma\neq\bm{1}}\int_{(\bm{2B})^\sigma}\left|K_\alpha(x,\vec{y})\prod_{i=1}^m(b_i(x)-b_i(y_i))- K_\alpha(z,\vec{y})\prod_{i=1}^m(b_i(z)-b_i(y_i))\right|\\
    &\qquad \times \prod_{i=1}^m|f_i(y_i)|\,d\vec{y}\\
    &\leq \sum_{\sigma\in S_m,\sigma\neq\bm{1}}\int_{(\bm{2B})^\sigma} |K_{\alpha}(x,\vec{y})-K_{\alpha}(z,\vec{y})|\prod_{i=1}^m|b_i(x)-b_i(y_i)|\,|f_i(y_i)|\,d\vec{y}\\
    &\qquad + \sum_{\sigma\in S_m,\sigma\neq\bm{1}}\int_{(\bm{2B})^\sigma} |K_{\alpha}(z,\vec{y})|\,\left|\prod_{i=1}^m(b_i(x)-b_i(y_i))-\prod_{i=1}^m(b_i(z)-b_i(y_i))\right|\\
    &\qquad \quad\times \prod_{i=1}^m|f_i(y_i)|\,d\vec{y}\\
    &=\sum_{\sigma\in S_m,\sigma\neq\bm{1}} (I_1^{\sigma}+I_2^{\sigma}).
\end{align*}
Fix $\sigma\in S_m,\sigma\neq \bf{1}$. Let us first estimate $I_1^\sigma$. By applying condition \eqref{eq: condicion de suavidad} we have that

\begin{align*}
  |K_\alpha(x,\vec{y})-K_\alpha(z,\vec{y})|&\lesssim \frac{|x-z|^\gamma}{(\sum_{i=1}^m |x-y_i|)^{mn-\alpha+\gamma}}\\
  &\lesssim\frac{|B|^{\gamma/n}}{(\sum_{i=1}^m |x_B-y_i|)^{mn-\alpha+\gamma}}.
\end{align*}
Therefore we have that
\begin{align*}
I_1^\sigma&\lesssim |B|^{\gamma/n}\prod_{i=1}^m\|b_i\|_{\Lambda(\delta)} \int_{(\bm{2B})^\sigma}\frac{(\sum_{i=1}^m|x_B-y_i|)^{m\delta}\prod_{i=1}^m|f_i(y_i)|}{(\sum_{i=1}^m|x_B-y_i|)^{mn-\alpha+\gamma}}\,d\vec{y}\\
&\lesssim\|\bm{b}\|_{(\Lambda(\delta))^m}^m|B|^{\gamma/n} \int_{(\bm{2B})^\sigma}\frac{\prod_{i=1}^m|f_i(y_i)|}{(\sum_{i=1}^m|x_B-y_i|)^{mn-\tilde\alpha+\delta+\gamma-\delta}}\,d\vec{y}\\
&\lesssim\|\bm{b}\|_{(\Lambda(\delta))^m}^m|B|^{\delta/n}\int_{(\bm{2B})^\sigma}\frac{\prod_{i=1}^m|f_i(y_i)|}{(\sum_{i=1}^m|x_B-y_i|)^{mn-\tilde\alpha+\delta}}\,d\vec{y},
\end{align*}
since $\gamma>\delta$ and $\vec{y}\in (\bm{2B})^{\sigma}$ implies that $|x_B-y_j|\gtrsim |B|^{1/n}$ for at least one index $1\leq j\leq m$. From this expression we can continue as in the estimate of $F(x_B,\sigma)$  performed in page~\pageref{pag: estimacion de F(x_B,sigma)} in order to obtain
\[I_1^\sigma \lesssim \|\bm{b}\|_{(\Lambda(\delta))^m}^m
|B|^{\tilde\delta/n-1}w^{-1}(B)
\prod_{i=1}^m\|f_iv_i\|_{p_i}.\]
\color{black}
Next we proceed to estimate $I_2^\sigma$. Fix $\sigma\in S_m, \sigma\neq\bm{1}$. By applying Lemma~\ref{lema: estimacion diferencia de productos} we have that
\begin{align*}
  \left|\prod_{i=1}^m(b_i(x)-b_i(y_i))-\prod_{i=1}^m(b_i(z)-b_i(y_i))\right|
  &\leq \sum_{j=1}^m |b_j(x)-b_j(z)|\prod_{i>j}|b_i(z)-b_i(y_i)|\prod_{i<j}|b_i(x)-b_i(y_i)|\\
  &\lesssim \left(\prod_{i=1}^m \|b_i\|_{\Lambda(\delta)}\right)|B|^{\delta/n}\sum_{j=1}^m\prod_{i>j}|z-y_i|^{\delta}\prod_{i<j}|x-y_i|^{\delta} 
\end{align*}
Since $x$ and $z$ belong to $B$ and $\vec{y}\in (\bm{2B})^\sigma$, we have that
\[|x-y_i|\lesssim \sum_{j=1}^m |x_B-y_j| \qquad \text{ and also }\qquad |z-y_i|\lesssim \sum_{j=1}^m |x_B-y_j|,\]
for each $i$, regardless $y_i$ belongs to $2B$ or $\mathbb{R}^n\backslash 2B$.
Therefore we arrive to
\[\left|\prod_{i=1}^m(b_i(x)-b_i(y_i))-\prod_{i=1}^m(b_i(z)-b_i(y_i))\right|\lesssim  \|\bm{b}\|_{(\Lambda(\delta))^m}^m |B|^{\delta/n}\left(\sum_{j=1}^m|x_B-y_j|\right)^{(m-1)\delta}.\]
By using this estimate we can proceed with $I_2^\sigma$ as follows
\begin{align*}
I_2^\sigma&\lesssim \|\bm{b}\|_{(\Lambda(\delta))^m}^m|B|^{\delta/n}\int_{(\bm{2B})^\sigma}\frac{\prod_{i=1}^m |f_i(y_i)|}{\left(\sum_{i=1}^m|x_B-y_i|\right)^{mn-\alpha+(1-m)\delta}}\,d\vec{y}\\
&\lesssim\|\bm{b}\|_{(\Lambda(\delta))^m}^m|B|^{\delta/n}\int_{(\bm{2B})^\sigma}\frac{\prod_{i=1}^m |f_i(y_i)|}{\left(\sum_{i=1}^m|x_B-y_i|\right)^{mn-\tilde\alpha+\delta}}\,d\vec{y}\\
&\lesssim\|\bm{b}\|_{(\Lambda(\delta))^m}^m|B|^{\delta/n}\prod_{i: \sigma_i=1}\int_{2B}\frac{|f_i(y_i)|}{|2B|^{1-\tilde\alpha_i/n+\delta/(mn)}}\,dy_i\prod_{i: \sigma_i=0}\int_{\mathbb{R}^n\backslash 2B}\frac{|f_i(y_i)|}{|x_B-y_i|^{n-\tilde\alpha_i+\delta/m}}\,dy_i.
\end{align*}
By applying Hölder inequality and condition \eqref{eq: condicion mezclada general} with $\lambda=2$ we get
\begin{align*}
    I_2^\sigma&\lesssim \|\bm{b}\|_{(\Lambda(\delta))^m}^m|B|^{\delta/n-\theta(\sigma)}\prod_{i=1}^m\|f_iv_i\|_{p_i}\prod_{i: \sigma_i=1}\|v_i^{-1}\mathcal{X}_{2B}\|_{p_i'}\prod_{i: \sigma_i=0}\left\|\frac{v_i^{-1}\mathcal{X}_{\mathbb{R}^n\backslash 2B}}{|x_B-\cdot|^{n-\tilde\alpha_i+\delta/m}}\right\|_{p_i'}\\
    &\lesssim \|\bm{b}\|_{(\Lambda(\delta))^m}^m |B|^{\tilde\delta/n -1}w^{-1}(B)
    \prod_{i=1}^m \|f_iv_i\|_{p_i},
\end{align*}
where $\theta(\sigma)=\sum_{i:\sigma_i=1}\left(1-\tilde\alpha_i/n+\delta/(mn)\right) $. So \eqref{eq: teo: acotacion Lp Lipschitz para T_alpha,b (producto) - eq3} holds and the proof is complete.
\end{proof}

\section{The class \texorpdfstring{$\mathcal{H}_m(\vec{p},\beta,\tilde\delta)$}{$Hm(p,\beta,\tilde\delta)$}}\label{seccion: clase de pesos}

We devote this section to study the class of weights $\mathcal{H}_m(\vec{p},\beta,\tilde\delta)$. We recall that a pair $(w,\vec{v})$ belongs to the class $\mathcal{H}_m(\vec{p},\beta,\tilde\delta)$ if there exists a positive constant $C$ such that

\begin{equation*}
    \frac{|B|^{1+(\delta-\tilde\delta)/n}}{w^{-1}(B)}\prod_{i=1}^m\left\|\frac{v_i^{-1}}{(|B|^{1/n}+|x_B-\cdot|)^{(n-\beta_i+\delta/m)}}\right\|_{p_i'}\leq C
\end{equation*}
for every ball $B=B(x_B, R)$.

In \cite{BPR22(2)} we gave a complete study of a variant of this class of weights. The following theorem establishes that for certain values of the parameters involved we get trivial functions on the considered class (cf. Theorem 1.2 in \cite{BPR22(2)}). 

\begin{teo}
Let $0<\beta<mn$, $\tilde\delta,\delta\in\mathbb{R}$ and $\vec{p}$ a vector of exponents. Then we have that: 
\begin{enumerate}[\rm(a)]
\item\label{teoNopesos-a} Given $\tilde\delta>\delta$ or $\tilde\delta>\beta-n/p$, if  $(w,\vec{v})\in\mathcal{H}_m(\vec{p},\beta,\tilde\delta)$, then there exists $1\leq i\leq m$ and a measurable set $E$ with $|E|>0$ such that $v_i=\infty$ on $E$.
\item\label{teoNopesos-b} The same conclusion holds when $\tilde\delta=\beta-n/p=\delta$.
\end{enumerate}
\end{teo}

\begin{proof}
    For item \eqref{teoNopesos-a}, let $\tilde\delta>\delta$ and assume that $(w,\vec{v})$ satisfies condition $\mathcal{H}_m(\vec{p},\beta,\tilde\delta)$. We choose $B(x_B,R)$ with $x_B$ being a Lebesgue point of $w^{-1}$. From \eqref{eq: condicion H_m(p,beta,delta tilde)-cal} we obtain 
    \begin{align*}
\prod_{i\in\mathcal{I}_1}\left\|\frac{v_i^{-1}}{(|B|^{1/n}+|x_B-\cdot|)^{n-\beta_i+\delta/m}}\right\|_\infty\,\prod_{i\in\mathcal{I}_2}\left(\int_{\mathbb{R}^n} \frac{v_i^{-p_i'}}{(|B|^{1/n}+|x_B-\cdot|)^{(n-\beta_i+\delta/m)p_i'}}\right)^{\tfrac{1}{p_i'}}&\lesssim \frac{w^{-1}(B)}{|B|R^{\delta-\tilde\delta}}.
	\end{align*}
 By letting $R\to 0$, we can conclude that
\begin{equation*}
\prod_{i\in\mathcal{I}_1}\left\|\frac{v_i^{-1}}{(|B|^{1/n}+|x_B-\cdot|)^{n-\beta_i+\delta/m}}\right\|_\infty\,\prod_{i\in\mathcal{I}_2}\left(\int_{\mathbb{R}^n} \frac{v_i^{-p_i'}}{(|B|^{1/n}+|x_B-\cdot|)^{(n-\beta_i+\delta/m)p_i'}}\right)^{\tfrac{1}{p_i'}}=0
\end{equation*}
which implies that there must exist $1\leq i \leq m$ such that  $v_i = \infty$ almost everywhere.

On the other hand, if $\tilde\delta>\beta-n/p$, we pick a ball $B$ such that $x_B$ is a Lebesgue point of $w^{-1}$ and of each $v^{-1}_i$.
Then condition \eqref{eq: condicion local} implies that
 	\[\prod_{i=1}^m \frac{1}{|B|}\int_B v_i^{-1}\leq \prod_{i\in\mathcal{I}_1}\left\|v_i^{-1}\mathcal{X}_B\right\|_\infty\prod_{i\in\mathcal{I}_2}\left( \frac{1}{|B|}\int_B v_i^{-p'_i} \right)^{1/p'_i }
  	\lesssim \frac{w^{-1}(B)}{|B|} R^{\tilde\delta - \beta + n/p}\]
	for every $R>0$. If we let again $R$ approach to zero, we obtain
	\begin{equation*}
	\prod_{i=1 }^{m} v_{i}^{-1}(x_B)=0, 
	\end{equation*}	
	so $\prod_{i=1 }^{m} v_{i}^{-1}=0$ almost everywhere. This allows us to conclude that the set $\bigcap_{i=1}^m \{v_i^{-1} >0 \}$ has null measure, so there must exist an index $1\leq j\leq m$ and a measurable set $E$ with $|E|>0$ such that $v_j=\infty$ on $E$.  
 
In order to prove item~\eqref{teoNopesos-b} we define 
\begin{equation*}
\frac{1}{\xi} = \sum_{i=1}^m \frac{1}{p_i'}=\frac{mp-1}{p}.
\end{equation*}
Then by Hölder inequality it follows that
\begin{equation*}
    \left(\int_{\mathbb{R}^n}\frac{(\prod_{i\in \mathcal{I}_2}v_i^{-1})^{\xi}}{(|B|^{1/n}+|x_B-y|)^{\sum_{i\in \mathcal{I}_2}(m-\beta_i+~\delta/m)\xi}}\right)^{1/\xi}
\leq 
\prod_{i\in \mathcal{I}_2}\left(\int_{\mathbb{R}^n}
\frac{v_i^{-p_i'}}{(|B|^{1/n}+|x_B-y|)^{(m-\beta_i+~\delta/m)\xi}}
\right)^{1/p_i'}
\end{equation*}
In this case, since $\delta=\tilde\delta$, condition $\mathcal{H}(\vec{p}, \beta, \tilde\delta)$  implies that 
\begin{equation*}
    \prod_{i\in \mathcal{I}_1}
    \left\|\frac{v_i^{-1}}{(|B|^{1/n}+|x_B -\cdot|)^{n-\beta_i+~\delta/m}} \right\|_\infty
    \left(\int_{\mathbb{R}^n}\frac{(\prod_{i\in \mathcal{I}_2}v_i^{-1})^{\xi}}{(|B|^{1/n}+|x_B-y|)^{\sum_{i\in \mathcal{I}_2}(m-\beta_i+~\delta/m)\xi}}\right)^{1/\xi}
\lesssim
\frac{w^{-1}(B)}{|B|}
\end{equation*}
and consequently
\begin{equation*}
        \left(\int_{\mathbb{R}^n}\frac{(\prod_{i=1 }^mv_i^{-1})^{\xi}}{(|B|^{1/n}+|x_B-y|)^{(mn-\beta+\delta)\xi}}\right)^{1/\xi}
\lesssim
\frac{w^{-1}(B)}{|B|}.
\end{equation*}
From the definition of $\xi$ and the fact that $\delta=\beta-n/p$, it follows that  $(mn-\beta+\delta)\xi=n$, and from here, we can use the same argument as in \cite{BPR22} (Theorem 1.2 (b)) to conclude that  $\prod_{i=1}^mv_i^{-1}=0$ almost everywhere and consequently that there exists $j$ satisfying $v_j=\infty$ almost everywhere. 
\end{proof}

As we previously said, we want to characterize the values of the parameters for which we have nontrivial pairs $(w,\vec{v})$ in $\mathcal{H}_m(\vec{p},\beta,\tilde\delta)$. In order to do so, we shall need the following auxiliary result, which can be proved by following the arguments given in the proof of Lemma~2.1 in \cite{BPR22(2)}. It states sufficient conditions under we have an equivalence between conditions \eqref{eq: condicion H_m(p,beta,delta tilde)-cal} and \eqref{eq: condicion global}.

\begin{lema}\label{lema: equivalencia con local y global}
Let $0<\beta<mn$, $\delta$ and $\tilde\delta$ be real numbers, $\vec{p}$ a vector of exponents and $(w,\vec{v})$ a pair of weights such that $v_i^{-1}\in\mathrm{RH}_\infty$ for $i\in\mathcal{I}_1$ and $v_i^{-p_i'}$ is doubling for $i\in\mathcal{I}_2$. Then, condition $\mathcal{H}_m(\vec{p},\beta,\tilde\delta)$ is equivalent to \eqref{eq: condicion global}.
\end{lema}

The following theorem allows us to describe the region where we can find nontrivial pairs in $\mathcal{H}_m(\vec{p},\beta,\tilde\delta)$ in terms of the parameters $p$, $\beta$ and $\tilde\delta$.
	
\begin{teo}\label{teo: ejemplos para la clase H_m(p,beta,delta)-cal}
	Given $\delta\in\mathbb{R}$ and $0<\beta<\min\{mn,mn+\delta\}$, there exist pairs of weights $(w,\vec{v})$ satisfying \eqref{eq: condicion H_m(p,beta,delta tilde)-cal} for every $\vec{p}$ and $\tilde\delta$ such that $\tilde\delta\leq \min\{\delta,\beta-n/p\}$, excluding the case $\tilde\delta=\delta$ when $\beta-n/p=\delta$.
\end{teo}

The figure below shows the colored area where we can find pairs of weights in $\mathcal{H}_m(\vec{p},\beta,\tilde\delta)$ and the values of the parameters involved that enclose it. Throughout this analysis $\delta$ is a fixed real number and the comparisons show the possible cases for $\beta$ with respect to $\delta$. 

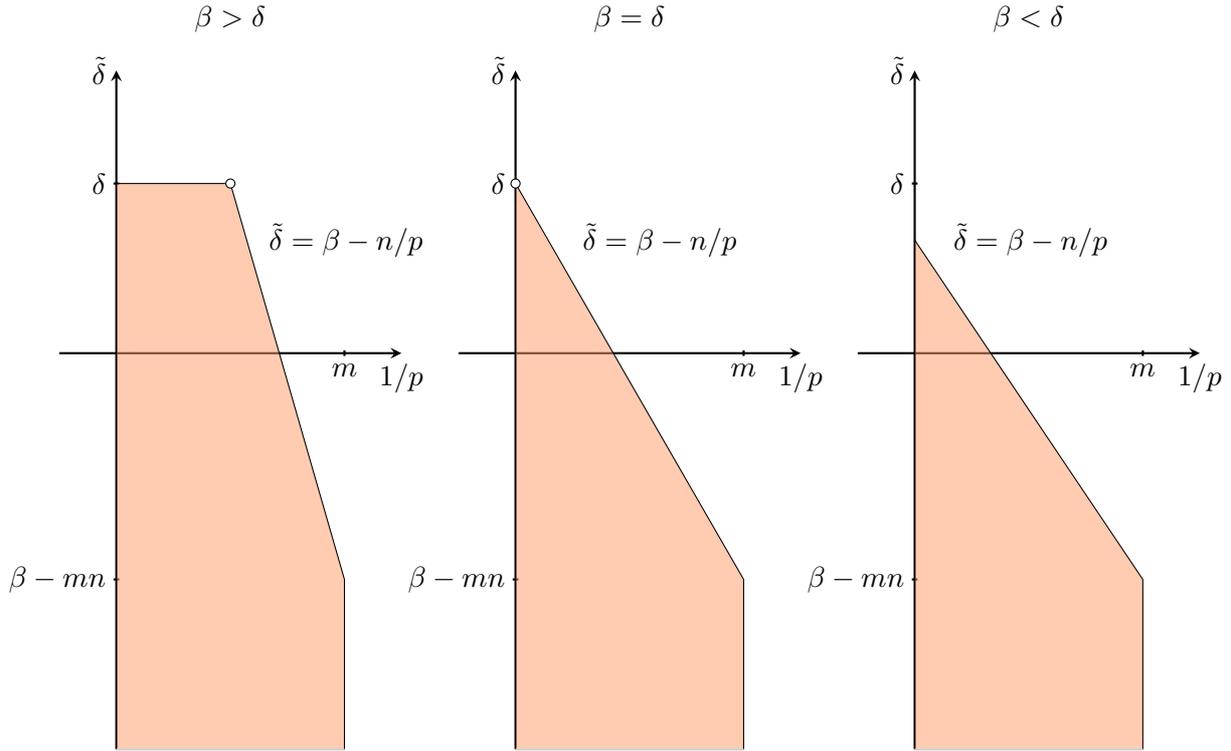
\begin{center} 
	\begin{tikzpicture}[scale=0.75]
	\node[above] at (-4,5.5) {$\beta>\delta$};
	\draw [-stealth, thick] (-6,-7)--(-6,5);
	\draw [-stealth, thick] (-7,0)--(-1,0);
	\draw [thick] (-6.05,3)--(-5.95, 3);
	\node [left] at (-6,5) {$\tilde\delta$};
	\node [left] at (-6,3) {$\delta$};
	\node [below] at (-1,0) {$1/p$};
	\node [below] at (-2,0) {$m$};
	\draw [thick] (-2,0.05)--(-2,-0.05);
	\draw [thick] (-6.05,-4)--(-5.95, -4);
	\node [left] at (-6,-4) {$\beta-mn$};
	\draw [fill=atomictangerine, fill opacity=0.5] (-6,-7)--(-6,3)--(-4,3)--(-2,-4)--(-2,-7)--cycle;
	\draw [color=white] (-2,-7)--(-6,-7);
	\draw [fill=white] (-4,3) circle (0.08cm);
	\node [right] at (-3.5,2) {$\tilde\delta=\beta-n/p$};
	\node[above] at (3,5.5) {$\beta=\delta$};
	\draw [-stealth, thick] (1,-7)--(1,5);
	\draw [-stealth, thick] (0,0)--(6,0);
	\draw [thick] (0.95,3)--(1.05, 3);
	\node [left] at (1,5) {$\tilde\delta$};
	\node [left] at (1,3) {$\delta$};
	\node [below] at (6,0) {$1/p$};
	\node [below] at (5,0) {$m$};
	\draw [thick] (5,0.05)--(5,-0.05);
	\draw [thick] (0.95,-4)--(1.05, -4);
	\node [left] at (1,-4) {$\beta-mn$};
	\draw [fill=atomictangerine, fill opacity=0.5] (1,-7)--(1,3)--(5,-4)--(5,-7)--cycle;
	\draw [color=white] (5,-7)--(1,-7);
	\draw [fill=white] (1,3) circle (0.08cm);
	\node [right] at (2,2) {$\tilde\delta=\beta-n/p$};
	\node[above] at (10,5.5) {$\beta<\delta$};
	\draw [-stealth, thick] (8,-7)--(8,5);
	\draw [-stealth, thick] (7,0)--(13,0);
	\draw [thick] (7.95,3)--(8.05, 3);
	\node [left] at (8,5) {$\tilde\delta$};
	\node [left] at (8,3) {$\delta$};
	\node [below] at (13,0) {$1/p$};
	\node [below] at (12,0) {$m$};
	\draw [thick] (12,0.05)--(12,-0.05);
	\draw [thick] (7.95,-4)--(8.05, -4);
	\node [left] at (8,-4) {$\beta-mn$};
	\draw [fill=atomictangerine, fill opacity=0.5] (8,-7)--(8,2)--(12,-4)--(12,-7)--cycle;
	\draw [color=white] (12,-7)--(8,-7);
	\node [right] at (8.5,2) {$\tilde\delta=\beta-n/p$};
	\end{tikzpicture}
\end{center}
\refstepcounter{BPR}\label{pag: regiones de pesos no triviales}

The proof of Theorem~\ref{teo: ejemplos para la clase H_m(p,beta,delta)-cal} requires a technical estimate of power functions that we state below.
 
\begin{lema}\label{lema: estimacion de la integral de |x|^a en una bola}
	Let $B=B(x_B,R)$ be a ball in $\mathbb{R}^n$ and $\alpha>-n$. Then 
	\[\int_B |x|^{\alpha}\,dx\approx R^n\left(\max\{R,|x_B|\}\right)^\alpha.\]
\end{lema} 

\medskip

\begin{proof}[Proof of Theorem~\ref{teo: ejemplos para la clase H_m(p,beta,delta)-cal}]
In \cite{BPRe} the authors exhibited examples of pair of weights in the class $\mathbb{H}_m(\vec{p},\beta, \tilde\delta)$. Since $\mathbb{H}_m(\vec{p},\beta, \tilde\delta)\subset \mathcal{H}_m(\vec{p},\beta, \tilde\delta)$, these examples cover the region $\beta-mn\leq \tilde\delta\leq \min\{\delta, \beta-n/p\}$, excluding the case $\delta=\beta-n/p$ when $\tilde\delta=\delta$ (see Theorem~6.6 in \cite{BPRe}). Therefore, it will be enough to exhibit examples when $\tilde\delta<\beta-mn$.

Recall that $\mathcal{I}_1=\{1\leq i\leq m: p_i=1\}$ and $\mathcal{I}_2=\{1\leq i\leq m: p_i>1\}$. We first assume that $\mathcal{I}_1\neq \emptyset$ and define $q_i=~n/p_i+(\delta-~\beta)/m$. Since $\beta<mn+\delta$, we can pick $-q_i<\tau_i<n/p_i'$ for every $i\in\mathcal{I}_2$ and $q_i<0$, and $0<\tau_i<n/p_i'$ if $i\in\mathcal{I}_2$ and $q_i\geq 0$. This election implies that 
\[\nu=\sum_{i\in\mathcal{I}_2,q_i\geq 0}\tau_i+\sum_{i\in\mathcal{I}_2,q_i<0}(\tau_i+q_i)>0.\]
We now choose  
\[0<\tau<\min\left\{\frac{\nu}{m_1}, n+\frac{\delta-\beta}{m}\right\},\]
and take $\tau_i=-\tau$ for every $i\in\mathcal{I}_1$. Let $\eta=\tilde\delta+\sum_{i=1}^m\tau_i+n/p-\beta$ and define
\[w(x)=|x|^\eta\quad\textrm{ and }\quad v_i(x)=|x|^{\tau_i},\quad \textrm{ for } 1\leq i\leq m.\]
 Observe that
 \[\eta=\tilde\delta+\sum_{i=1}^m\tau_i+n/p-\beta<\tilde\delta+\sum_{i=1}^m\frac{n}{p_i'}+\frac{n}{p}-\beta=\tilde\delta+mn-\beta<0,\]
 since $\tilde\delta<\beta-mn$, so $w^{-1}$ is a locally integrable function. On the other hand, $v_i^{-1}\in \mathrm{RH}_\infty$  for  $i\in\mathcal{I}_1$ by virtue of Remark~\ref{obs: potencias positivas en RH inf}, so the same conclusion holds for these weights. For $i\in\mathcal{I}_2$ we also have that $v_i^{-p_i'}$ is locally integrable since
 $\tau_i<n/p_i'$. Consequently, by Lemma~\ref{lema: equivalencia con local y global}, it will be enough to show that condition \eqref{eq: condicion global} holds, that is, we need to check that there exists a positive constant $C$ such that the inequality
\begin{equation}\label{eq: teo: ejemplos para Hcal - eq1}
\frac{|B|^{1+(\delta-\tilde\delta)/n}}{w^{-1}(B)}\prod_{i\in\mathcal{I}_1}\left\|\frac{v_i^{-1}\mathcal{X}_{\mathbb{R}^n\backslash B}}{|x_B-\cdot|^{n-\beta/m+\delta/m}}\right\|_\infty\,\prod_{i\in\mathcal{I}_2}\left(\int_{\mathbb{R}^n\backslash B}\frac{v_i^{-p_i'}}{|x_B-\cdot|^{(n-\beta/m+\delta/m)p_i'}}\right)^{1/p_i'}\leq C
\end{equation}
holds for every ball $B=B(x_B,R)$. 

We shall first consider the case $|x_B|\leq R$. By Lemma~\ref{lema: estimacion de la integral de |x|^a en una bola} we obtain
\begin{equation}\label{eq: teo: ejemplos para Hcal - eq2}
\frac{|B|^{1+(\delta-\tilde\delta)/n}}{w^{-1}(B)}\lesssim R^{\delta-\tilde\delta+\eta}.
\end{equation}
Also notice that, if $i\in\mathcal{I}_1$ and $B_k=B(x_B,2^kR)$, for $k\in\mathbb{N}$, we have
\begin{align*}
\left\|\frac{v_i^{-1}\mathcal{X}_{\mathbb{R}^n\backslash B}}{|x_B-\cdot|^{n-\beta/m+\delta/m}}\right\|_\infty&\lesssim \sum_{k=0}^\infty \left\|\frac{v_i^{-1}\mathcal{X}_{B_{k+1}\backslash B_k}}{|x_B-\cdot|^{n-\beta/m+\delta/m}}\right\|_\infty\\
&\lesssim \sum_{k=0}^\infty \left(2^kR\right)^{-\tau_i-n+\beta/m-\delta/m}\\
&\lesssim   R^{-\tau_i-n+\beta/m-\delta/m},
\end{align*}
since $-\tau_i-n+\beta/m-\delta/m<0$. This yields
\begin{equation}\label{eq: teo: ejemplos para Hcal - eq3}
\prod_{i\in\mathcal{I}_1}\left\|\frac{v_i^{-1}\mathcal{X}_{\mathbb{R}^n\backslash B}}{|x_B-\cdot|^{n-\beta/m+\delta/m}}\right\|_\infty\lesssim R^{-\sum_{i\in\mathcal{I}_1}(\tau_i+q_i)}.
\end{equation}
Finally, since $\tau_i+q_i>0$ for every $i\in\mathcal{I}_2$, Lemma~\ref{lema: estimacion de la integral de |x|^a en una bola} allows us to get 
\begin{align*}
\left(\int_{\mathbb{R}^n\backslash B} \frac{v_i^{-p_i'}(y)}{|x_B-y|^{(n-\beta/m+\delta/m)p_i'}}\,dy\right)^{1/p_i'}&\lesssim \sum_{k=0}^\infty (2^kR)^{-n+\beta/m-\delta/m}\left(\int_{B_{k+1}\backslash B_k} |y|^{-\tau_ip_i'}\,dy\right)^{1/p_i'}\\
	&\lesssim \sum_{k=0}^\infty (2^kR)^{-n+\beta/m-\delta/m-\tau_i+n/p_i'}\\
	&\lesssim R^{-n/p_i+\beta/m-\delta/m-\tau_i}.
\end{align*}
This yields
\begin{equation}\label{eq: teo: ejemplos para Hcal - eq4}
\prod_{i\in\mathcal{I}_2}\left(\int_{\mathbb{R}^n\backslash B}\frac{v_i^{-p_i'}(y)}{|x_B-y|^{(n-\beta/m+\delta/m)p_i'}}\,dy\right)^{1/p_i'}\lesssim R^{-\sum_{i\in\mathcal{I}_2}(\tau_i+q_i)}.
\end{equation}
By combining \eqref{eq: teo: ejemplos para Hcal - eq2}, \eqref{eq: teo: ejemplos para Hcal - eq3} and \eqref{eq: teo: ejemplos para Hcal - eq4}, the left-hand side of \eqref{eq: teo: ejemplos para Hcal - eq1} is bounded by a multiple constant of 
\[R^{\delta-\tilde\delta+\eta-\sum_{i=1}^m(q_i+\tau_i)}\approx 1.\] 
Let us now consider the case $|x_B|>R$. Again, by Lemma~\ref{lema: estimacion de la integral de |x|^a en una bola}, we have that
\begin{equation}\label{eq: teo: ejemplos para Hcal - eq5}
\frac{|B|^{1+(\delta-\tilde\delta)/n}}{w^{-1}(B)}\lesssim R^{\delta-\tilde\delta}|x_B|^\eta\lesssim R^{\delta-\tilde\delta+\eta},
\end{equation}
because $\eta<0$. Since $|x_B|>R$, there exists a number $N\in\mathbb{N}$ such that $2^NR<|x_B|\leq 2^{N+1}R$. When $i\in\mathcal{I}_1$ we write
\begin{align*}
\left\|\frac{v_i^{-1}\mathcal{X}_{\mathbb{R}^n\backslash B}}{|x_B-\cdot|^{n-\beta/m+\delta/m}}\right\|_\infty&\lesssim \sum_{k=0}^N \left\|\frac{v_i^{-1}\mathcal{X}_{B_{k+1}\backslash B_k}}{|x_B-\cdot|^{n-\beta/m+\delta/m}}\right\|_\infty+\sum_{k=N+1}^\infty \left\|\frac{v_i^{-1}\mathcal{X}_{B_{k+1}\backslash B_k}}{|x_B-\cdot|^{n-\beta/m+\delta/m}}\right\|_\infty\\
&=S_1^i+S_2^i.
\end{align*}
We first observe that
\[S_1^i\lesssim |x_B|^{-\tau_i}\sum_{k=0}^N \left(2^kR\right)^{-n+\beta/m-\delta/m}\lesssim |x_B|^{-\tau_i}R^{-n+\beta/m-\delta/m}=|x_B|^{-\tau_i}R^{-q_i}\]
and
\begin{align*}
S_2^i\lesssim \sum_{k=N+1}^\infty \left(2^kR\right)^{-\tau_i-n+\beta/m-\delta/m}
&\lesssim \left(2^{N}R\right)^{-\tau_i-n+\beta/m-\delta/m}\sum_{k=0}^\infty 2^{k(-\tau_i-n+\beta/m-\delta/m)}\\
&\lesssim |x_B|^{-\tau_i}R^{-n+\beta/m-\delta/m}\\
&=|x_B|^{-\tau_i}R^{-q_i}.
\end{align*}
The corresponding bounds for $S_1^i$ and $S_2^i$ lead to
\begin{equation}\label{eq: teo: ejemplos para Hcal - eq6}
\prod_{i\in\mathcal{I}_1}\left\|\frac{v_i^{-1}\mathcal{X}_{\mathbb{R}^n\backslash B}}{|x_B-\cdot|^{n-\beta/m+\delta/m}}\right\|_\infty\lesssim |x_B|^{-\sum_{i\in\mathcal{I}_1}\tau_i}\,\,R^{-\sum_{i\in\mathcal{I}_1}q_i}.
\end{equation}

For $i\in\mathcal{I}_2$ we proceed in a similar way by splitting the integral as follows
\begin{align*}
\left(\int_{\mathbb{R}^n\backslash B} \frac{v_i^{-p_i'}(y)}{|x_B-y|^{(n-\gamma/m+1/m)p_i'}}\,dy\right)^{1/p_i'}&\lesssim \sum_{k=0}^\infty(2^{k}R)^{-n+\gamma/m-1/m}\left(\int_{B_k} |y|^{-\beta_ip_i'}\,dy\right)^{1/p_i'}\\
&=\sum_{k=0}^N+\sum_{k=N+1}^\infty\\
	&=S_1^i+S_2^i.
\end{align*}\label{pag: estimacion del producto para i fuera de I_1, |x_B|>R}
 We estimate the sum $S_1^i+S_2^i$ by distinguishing into the cases $q_i<0$, $q_i=0$ and $q_i>0$. When $q_i<0$, by Lemma~\ref{lema: estimacion de la integral de |x|^a en una bola} we obtain
\begin{align*}
	S_1^i&\lesssim \sum_{k=0}^N(2^{k}R)^{-n+\beta/m-\delta/m+n/p_i'}|x_B|^{-\tau_i}\\
	&\lesssim |x_B|^{-\tau_i}R^{-q_i}\sum_{k=0}^N 2^{-kq_i}\\
	&\lesssim |x_B|^{-\tau_i}(2^NR)^{-q_i}\\
	&\lesssim |x_B|^{-\tau_i-q_i},
\end{align*}
because we assumed $q_i<0$. For $S_2^i$ we apply again Lemma~\ref{lema: estimacion de la integral de |x|^a en una bola} in order to get
\begin{align*}
	S_2^i&\lesssim \sum_{k=N+1}^\infty(2^{k}R)^{-n+\beta/m-\delta/m+n/p_i'-\tau_i}\\
	&\lesssim \sum_{k=N+1}^\infty \left(2^{k}R\right)^{-\tau_i-q_i}\\
	&= \left(2^{N+1}R\right)^{-\tau_i-q_i}\sum_{k=0}^\infty 2^{-k(\tau_i+q_i)}\\
	&\lesssim |x_B|^{-\tau_i-q_i},
\end{align*}
since $q_i+\tau_i>0$. This yields
\begin{equation}\label{eq: teo: ejemplos para Hcal - eq7}
S_1^i+S_2^i\lesssim |x_B|^{-\tau_i-q_i}
\end{equation}
when $q_i<0$.

We now assume that $q_i=0$. By proceeding in a similar way as above, we obtain
\[S_1^i\lesssim |x_B|^{-\tau_i}N\lesssim |x_B|^{-\tau_i}\log_2\left(\frac{|x_B|}{R}\right),\]
and
\[S_2^i\lesssim |x_B|^{-\tau_i}\]
since $\tau_i>0$ when $q_i=0$. Therefore,
\begin{equation}\label{eq: teo: ejemplos para Hcal - eq8}
	S_1^i+S_2^i\lesssim |x_B|^{-\tau_i}\left(1+\log_2\left(\frac{|x_B|}{R}\right)\right)\lesssim |x_B|^{-\tau_i}\log_2\left(\frac{|x_B|}{R}\right).
\end{equation}

Finally, assume that $q_i>0$. For $S_2^i$ we proceed exactly as we did for the case $q_i<0$ and get the same bound. On the other hand, for $S_1^i$ we have that
\begin{align*}
	S_1^i&\lesssim \sum_{k=0}^N(2^{k}R)^{-n+\beta/m-\delta/m+n/p_i'}|x_B|^{-\tau_i}\\
	&\lesssim |x_B|^{-\tau_i}R^{-q_i}\sum_{k=0}^N 2^{-kq_i}\\
	&\lesssim |x_B|^{-\tau_i}\left(2^NR\right)^{-q_i}2^{Nq_i}\\
	&\lesssim |x_B|^{-\tau_i-q_i}2^{Nq_i}.
\end{align*}\label{pag: estimacion de S_1^i y S_2^i,  theta_i>0}
Therefore, if $i\in\mathcal{I}_2$ and $q_i>0$, we get
\begin{equation}\label{eq: teo: ejemplos para Hcal - eq9}
	S_1^i+S_2^i\lesssim |x_B|^{-\tau_i-q_i}\left(1+2^{Nq_i}\right)\lesssim 2^{Nq_i}|x_B|^{-\tau_i-q_i}. 
\end{equation}
By combining the estimates in \eqref{eq: teo: ejemplos para Hcal - eq7},\eqref{eq: teo: ejemplos para Hcal - eq8} and \eqref{eq: teo: ejemplos para Hcal - eq9}  we obtain
\begin{align*}	\prod_{i\in\mathcal{I}_2}\left(\int_{\mathbb{R}^n\backslash B} \frac{v_i^{-p_i'}(y)}{|x_B-y|^{(n-\beta/m+\delta/m)p_i'}}\,dy\right)^{1/p_i'}&\lesssim \prod_{i\in\mathcal{I}_2, q_i<0} |x_B|^{-\tau_i-q_i} \prod_{i\in\mathcal{I}_2, q_i=0} |x_B|^{-\tau_i}\log_2\left(\frac{|x_B|}{R}\right) \\
	&\qquad\times\prod_{i\in\mathcal{I}_2, q_i>0} |x_B|^{-\tau_i-q_i}2^{Nq_i} \\
	&\lesssim |x_B|^{-\sum_{i\in\mathcal{I}_2}(\tau_i+q_i)}2^{N\sum_{i\in\mathcal{I}_2,q_i> 0}q_i}\\
	&\qquad \times\left(\log_2\left(\frac{|x_B|}{R}\right)\right)^{\#\{i\in\mathcal{I}_2, q_i=0\}}.
\end{align*}

The estimate above combined with \eqref{eq: teo: ejemplos para Hcal - eq5} and \eqref{eq: teo: ejemplos para Hcal - eq6} allows us to bound the left-hand side of \eqref{eq: teo: ejemplos para Hcal - eq1} by a multiple constant of
\[R^{\delta-\tilde\delta+\eta} |x_B|^{-\sum_{i\in\mathcal{I}_1}\tau_i}R^{-\sum_{i\in\mathcal{I}_1}q_i}|x_B|^{-\sum_{i\in\mathcal{I}_2}(\tau_i+q_i)}2^{N\sum_{i\in\mathcal{I}_2,q_i> 0}q_i}
\left(\log_2\left(\frac{|x_B|}{R}\right)\right)^{\#\{i\in\mathcal{I}_2, q_i=0\}}\]
or equivalently by
\begin{equation}\label{eq: teo: ejemplos para Hcal - eq10}
\left(\frac{R}{|x_B|}\right)^{\delta-\tilde\delta+\eta-\sum_{i\in\mathcal{I}_1}q_i-\sum_{i\in\mathcal{I}_2,q_i> 0}q_i}\left(\log_2\left(\frac{|x_B|}{R}\right)\right)^{\#\{i\in\mathcal{I}_2, q_i=0\}}.	
\end{equation}
Since $\sum_{i=1}^m q_i=n/p+\delta-\beta$ and $\eta=\tilde\delta+\sum_{i=1}^m\tau_i+n/p-\beta$, then the exponent of $R/|x_B|$ is equal to 
\begin{align*}
	\delta-\tilde\delta+\eta-\sum_{i\in\mathcal{I}_1}q_i-\sum_{i\in\mathcal{I}_2,q_i> 0}q_i&=\delta+\sum_{i=1}^m\tau_i+n/p-\beta-\sum_{i\in\mathcal{I}_1}q_i-\sum_{i\in\mathcal{I}_2,q_i>0}q_i\\
	&=\sum_{i\in\mathcal{I}_2,q_i<0}(\tau_i+q_i)+\sum_{i\in\mathcal{I}_1}\tau_i+\sum_{i\in\mathcal{I}_2,q_i\geq 0}\tau_i\\
	&=\nu-m_1\tau,
\end{align*}
which is positive from our election of $\tau$. Since $\log t\lesssim \varepsilon^{-1}t^\varepsilon$ for every $t\geq 1$ and every $\varepsilon>0$, we can bound \eqref{eq: teo: ejemplos para Hcal - eq10} by a multiple constant of
\[\left(\frac{R}{|x_B|}\right)^{\nu-m_1\tau-\varepsilon\#\{i\in\mathcal{I}_2, q_i=0\}},\]
and this exponent is positive provided we choose $\varepsilon>0$ sufficiently small. The proof is complete when $\mathcal{I}_1\neq\emptyset$. Otherwise, we can follow the same steps and define the same parameters, omitting the factor corresponding to $\mathcal{I}_1$. This concludes the proof. 
\end{proof}

We finish this section with the proof of Theorem~\ref{teo: consecuencia de pesos iguales}.

\begin{proof}[Proof of Theorem~\ref{teo: consecuencia de pesos iguales}]
 Since $\vec{v}\in\mathcal{H}_m(\vec{p},\beta,\tilde\delta)$, condition \eqref{eq: condicion local} implies that 
\begin{equation}\label{eq: teo: caso de pesos iguales - eq1}
|B|^{-\tilde\delta/n+\beta/n-1/p}\prod_{i\in\mathcal{I}_1}\|v_i^{-1}\mathcal{X}_B\|_\infty\,\prod_{i\in\mathcal{I}_2}\left(\frac{1}{|B|}\int_B v_i^{-p_i'}\right)^{1/p_i'}\leq \frac{C}{|B|}\int_B \prod_{i=1}^m v_i^{-1}.
\end{equation}
Notice that $\sum_{i=1}^m \xi/p_i'=1$. Then we can apply Hölder inequality with $p_i'/\xi$ in order to get 
\[\left(\frac{1}{|B|}\int_B \left(\prod_{i=1}^m v_i^{-1}\right)^\xi\right)^{1/\xi}\leq \prod_{i\in\mathcal{I}_1}\|v_i^{-1}\mathcal{X}_B\|_\infty\,\prod_{i\in\mathcal{I}_2}\left(\frac{1}{|B|}\int_B v_i^{-p_i'}\right)^{1/p_i'}.\]
By multiplying each side of this inequality by $|B|^{-\tilde\delta/n+\beta/n-1/p}$ and using \eqref{eq: teo: caso de pesos iguales - eq1} we arrive to
\[|B|^{-\tilde\delta/n+\beta/n-1/p}\left(\frac{1}{|B|}\int_B \left(\prod_{i=1}^m v_i^{-1}\right)^\xi\right)^{1/\xi}\leq \frac{C}{|B|}\int_B \prod_{i=1}^m v_i^{-1}.\]
Therefore, we can conclude that
\[|B|^{-\tilde\delta/n+\beta/n-1/p}\leq C\]
for every ball $B$, since $\xi>1$. Then we must have that $\tilde\delta/n=\beta/n-1/p$, as desired.	
\end{proof}

\section*{Acknowledgements}

We would like to specially thanks to Ph. D. Gladis Pradolini for suggesting us these problems, as well as giving us useful advices for redaction and bibliography. 





\def\cprime{$'$}
\providecommand{\bysame}{\leavevmode\hbox to3em{\hrulefill}\thinspace}
\providecommand{\MR}{\relax\ifhmode\unskip\space\fi MR }
\providecommand{\MRhref}[2]{%
  \href{http://www.ams.org/mathscinet-getitem?mr=#1}{#2}
}
\providecommand{\href}[2]{#2}

\end{document}